\documentclass[11pt,a4paper,reqno]{amsart}

\usepackage[all]{xy}
\usepackage{amsthm}
\usepackage{amssymb} 
\usepackage{amsmath,amscd} 
\usepackage{mathrsfs}
\usepackage{color}
\usepackage{enumerate}
\usepackage{longtable}

\newcommand{\Dbmod}[1]{\sfD^{\mrb}(\operatorname{mod} #1 )}

\newcommand{\Kbproj}[1]{\sfK^{\mrb}( #1 \operatorname{proj})}

\newcommand{\agl}[1]{\langle #1 \rangle}

\def\twotilt{\operatorname{2-tilt}}
\def\2silt{\operatorname{2-silt}}

\def\silt{\operatorname{silt}}
\def\tilt{\operatorname{tilt}}
\def\sttilt{\textup{s}\tau\textup{-}\operatorname{tilt}}
\def\grsttilt{\textup{grs}\tau\operatorname{tilt}}

\def\basic{\mathsf{basic}}
\def\stab{\operatorname{stab}}
\def\rad{\operatorname{rad}}

\def\frkm{{\mathfrak{m}}}

\def\frkS{{\mathfrak{S}}}

\def\turn!{\textup{!`}}

\def\chara{\operatorname{char}}
\def\sgn{\operatorname{sgn}}

\def\ad{\operatorname{ad}}

\def\Ind{\operatorname{Ind}}

\def\Res{\operatorname{Res}}
\def\ind{\mathop{\mathrm{ind}}\nolimits}

\def\grmod{\operatorname{grmod}}

\def\mrb{\mathrm{b}}

\def\thick{\mathop{\mathsf{thick}}\nolimits}

\def\kk{{\mathbf k}}

\def\AA{{\Bbb A}}

\def\BB{{\Bbb B}}
\def\CC{{\Bbb C}}

\def\DD{{\Bbb D}}

\def\EE{{\Bbb E}}

\def\FF{{\Bbb F}}
\def\GG{{\Bbb G}}

\def\LL{{\Bbb L}}
\def\NN{{\Bbb N}}

\def\XX{{\Bbb X}}

\def\ZZ{{\Bbb Z}}

\def\cC{{\mathcal C}}

\def\cT{{\mathcal T}}

\def\sfD{{\mathsf{D}}}
\def\sfE{{\mathsf{E}}}
\def\sfF{{\mathsf{F}}}

\def\sfK{{\mathsf{K}}}

\def\tuD{{\textup{D}}}

\def\tuH{{\textup{H}}}

\def\tuK{{\textup{K}}}

\def\id{\operatorname{id}}

\def\mod{\operatorname{mod}}

\def\Ker{\mathop{\mathrm{Ker}}\nolimits}

\def\proj{\operatorname{proj}}

\def\add{\operatorname{add}}

\def\Coker{\operatorname{Cok}}

\def\Hom{\operatorname{Hom}}

\def\End{\operatorname{End}}
\def\Aut{\operatorname{Aut}}

\def\Ext{\operatorname{Ext}}

\newtheorem{lemma}{Lemma}[section]
\newtheorem{proposition}[lemma]{Proposition}
\newtheorem{theorem}[lemma]{Theorem}
\newtheorem{corollary}[lemma]{Corollary}

\theoremstyle{definition}
\newtheorem{remark}[lemma]{Remark}
\newtheorem{example}[lemma]{Example}

\newtheorem{definition}[lemma]{Definition}

\theoremstyle{remark}

\title[skew group algebra extensions]
{ 
$\tau$-tilting theory and  silting theory of  skew group algebra extensions}

\dedicatory{Dedicated to Jun-ichi Miyachi's 65th birthday}

\author[Kimura, Koshio, Kozakai, Minamoto, Mizuno]{Yuta Kimura, Ryotaro Koshio, Yuta Kozakai, Hiroyuki Minamoto, Yuya Mizuno}

\address{Yuta Kimura : Department of Mechanical Systems Engineering, Faculty of Engineering, Hiroshima Institute of Technology, 2-1-1 Miyake, Saeki-ku Hiroshima 731-5143, Japan}
\email{y.kimura.4r@cc.it-hiroshima.ac.jp}

\address{Ryotaro Koshio: Department of Mathematics, Graduate School of Science, Tokyo University of Science 1-3, Kagurazaka, Shinjuku-ku, Tokyo, 162-8601, Japan} 
\email{koshio@rs.tus.ac.jp}

\address{Yuta Kozakai: Department of Mathematics, Tokyo University of Science 1-3, Kagurazaka, Shinjuku-ku, Tokyo, 162-8601, Japan} 
\email{kozakai@rs.tus.ac.jp}

\address{Hiroyuki Minamoto: Department of Mathematics, Graduate School of Science/ Faculty of Science, Osaka Metropolitan University
1-1 Gakuen-cho, Nakaku, Sakai, Osaka 599-8531, Japan}
\email{minamoto@omu.ac.jp}

\address{Yuya Mizuno: Faculty of Liberal Arts, Sciences and Global Education / Graduate School of Science, Osaka Metropolitan University, 1-1 Gakuen-cho, Naka-ku, Sakai, Osaka 599-8531, Japan}
\email{yuya.mizuno@omu.ac.jp}


\begin{document}
%
%
%

\maketitle

\begin{abstract}
Let $\Lambda$ be a finite dimensional algebra with an action by a finite group $G$ 
and $A:= \Lambda *G$ the skew group algebra. 
One of our main results asserts that  the canonical restriction-induction adjoint pair of the skew group algebra extension $\Lambda \subset A$ 
induces a poset isomorphism between the poset of   $G$-stable support $\tau$-tilting modules over $\Lambda$ and 
that of \textup{(}$\mod G$\textup{)}-stable support $\tau$-tilting modules over $A$. 
We also establish  a similar poset isomorphism of posets of  appropriate classes of silting complexes over $\Lambda$ and $A$. 
These two results  generalize and  unify  preceding results by Zhang-Huang, Breaz-Marcus-Modoi and the second and the third authors. 
Moreover, we give a practical  condition under which $\tau$-tilting finiteness and silting discreteness of $\Lambda$ are inherited to those of $A$.  
As applications we study $\tau$-tilting theory and silting theory of the (generalized) preprojective algebras and the folded mesh algebras. 
Among other things,  we determine the posets of support $\tau$-tilting modules and of silting complexes over 
preprojective algebra $\Pi(\LL_{n})$ of type $\LL_{n}$.  
\end{abstract}

\tableofcontents

\section{Introduction}

A decade has been passed since the introduction of $\tau$-tilting theory \cite{AIR} and silting theory \cite{AI}. 
During these years, those theories have been extensively studied in many researchers 
and now become basic and fundamental subjects of representation theory of algebras. 
These theories study hidden symmetry of the module category and the derived category of algebras 
by using specific classes of objects, $\tau$-tilting modules and silting complexes. 
Thus, these theories have connections to various areas of mathematics. 
Among other things, we point out that since silting theory is strongly related to $t$-structures of the derived category, 
 it is effectively used to study of Bridgeland's stability conditions \cite{Bridgeland}. 
 
The central aim of the present paper is 
to study $\tau$-tilting theory and silting theory of the skew group algebra extension 
$\Lambda \subset \Lambda *G$. 
Namely, given an algebra $\Lambda$ with an action $G\curvearrowright \Lambda$ by a finite group $G$, 
we study relationships between $\tau$-tilting modules and silting complexes over $\Lambda$ and 
the skew group algebra $A:= \Lambda * G$ via the induction and the restriction functor. 
A detailed explanation is given in the next section, in which the main results are summarized as Theorem \ref{202404091842}. 
Another important result proves that  under  a practical  condition, $\tau$-tilting finiteness and silting discreteness of $\Lambda$ are inherited. 
It is summarized in Theorem \ref{202404101701}.  
In Section \ref{202406031540}, we  review recent related results which study a triangulated category $\cT$ 
with an action $G \curvearrowright \cT$ by a finite group $G$, and then we explain relationships between these results and our results.

As applications we study $\tau$-tilting theory and silting theory of the (generalized) preprojective algebras and the folded mesh algebras in Section \ref{202404102228}. 
Among other things,  we determine the posets of support $\tau$-tilting modules and of silting complexes of 
preprojective algebra $\Pi(\LL_{n})$ of type $\LL_{n}$.

\subsection{$\tau$-tilting theory and silting theory of skew group algebra extensions}\label{202406031522}

Let $\Lambda$ be a finite dimensional algebra, $G$ a finite group acting on $\Lambda$ and $A:= \Lambda * G$ the skew group algebra. 
The canonical inclusion $\Lambda \subset A$  induces the  adjoint pair  $(\Ind_{\Lambda}^{A}, \Res_{\Lambda}^{A})$ 
of the induction functor $\Ind_{\Lambda}^{A} = A \otimes_{\Lambda} -$ and 
the restriction functor $\Res_{\Lambda}^{A}$ of the categories of finite dimensional (left) modules.
\[
\Ind^{A}_{\Lambda}: \mod \Lambda \rightleftarrows \mod A: \Res_{\Lambda}^{A}.
\]
The one of the main aims of this paper is to compare 
support $\tau$-tilting modules $M$ over $\Lambda$ and those $N$ of over $A$, via the adjoint pair  $(\Ind_{\Lambda}^{A}, \Res_{\Lambda}^{A})$.
Namely we ask which support $\tau$-tilting modules over $\Lambda$ (resp. $A$) 
are sent to those over $A$ (resp. $\Lambda$) by the functor $\Ind_{\Lambda}^{A}$ (resp. $\Res_{\Lambda}^{A}$).

Since it is  a natural question, there are several previous research of this problem (or of equivalent one).

\subsubsection{Preceding results}

Zhang-Huang \cite{ZH} studied a skew group algebra extension $\Lambda \subset A$ where $A:= \Lambda * G$ under the assumption that $\chara \kk \nmid |G|$. 
Their main result proved that  
the induction functor $\Ind_{\Lambda}^{A}$ sends 
$G$-stable support $\tau$-tilting modules $M$ over $\Lambda$  
to support $\tau$-tilting modules $\Ind_{\Lambda}^{A}(M)$ over $A$ 
and that this map  is injective (for the notion  $G$-stability see Definition \ref{202309151111}). 
They also gave a restriction on the image of the map induced by  $\Ind_{\Lambda}^{A}$.   
But they could not identify the image of this map.

Breaz-Marcus-Modoi \cite{BMM} dealt with 
 the extension  $\Lambda \subset A$ of a  strongly $G$-graded algebra $A=\bigoplus_{g\in G} A_{g}$ and its degree $e_{G}$-part $\Lambda := A_{e_{G}}$ and compared support $\tau$-tilting modules of $\Lambda$ and $A$ via the canonical adjoint pair.  
(We note that
as is explained in Appendix \ref{202404011608}, 
this problem is equivalent to  
that of skew group algebra extension by a finite group $G$,  
via Cohen-Montgomery duality.)
They worked with the assumption that $\Lambda$ is self-injective. 
They showed that 
the induction functor $\Ind_{\Lambda}^{A}$ sends 
$G$-stable support $\tau$-tilting modules $M$ over $\Lambda$  
to support $\tau$-tilting modules $\Ind_{\Lambda}^{A}(M)$ over $A$. 
 Moreover, under the additional assumption $\chara \kk \nmid |G|$, 
 they also find a condition for a support $\tau$-tilting module $N$ over $A$ that ensures that $\Res_{\Lambda}^{A}(N)$ is 
 a support $\tau$-tilting modules over $\Lambda$.

Finally, 
the second author and the third author of this paper \cite{KK} dealt with a group extension $N \triangleleft H$.
 They compared support $\tau$-tilting modules of $\Lambda:= \kk N$ and $A:=\kk H$ of group algebras 
 via the canonical adjoint pair induced from  the extension $\Lambda =\kk N \subset \kk H = A$. 
 It is noteworthy that they did not  put  any hypothesis on the characteristic of the base field $\kk$. 
(We note that 
 $A=\kk H$ has a canonical structure of a strongly $G:=H/N$-graded algebras such that $A_{e_G} = \kk N$, see Appendix \ref{202404011608}.) 
They proved that the functor $\Ind_{\Lambda}^{A}$ induces an injective map 
from the poset of $G$-stable support $\tau$-tilting modules over $\Lambda$ to the poset of those over $A$. 
This is analogous to the aforementioned result by Zhang-Huang.
However,  they additionally succeeded to identify the image of the map induced from the functor $\Ind_{\Lambda}^{A}$.

\subsubsection{Poset isomorphisms}

In the paper we deal with  
the skew group algebra extension $\Lambda \subset \Lambda* G$ without any hypothesis on the characteristic of $\kk$, 
and obtain   results that unifies and generalizes previous results recalled above.
One of the novelties of the paper is the notion of \emph{\textup{(}$\mod G$\textup{)}-stability} of a module $N$ over $A$ (Definition \ref{202309171251}).  
This notion precisely identifies the image of the map induced by $\Ind_{\Lambda}^{A}$. 
Namely we prove that the functor  $\Ind_{\Lambda}^{A}$   
induces  a  poset isomorphism between the poset 
$(\sttilt \Lambda)^{G}$ of $G$-stable support $\tau$-tilting modules 
and the poset $(\sttilt A)^{\mod G}$ of \textup{(}$\mod G$\textup{)}-stable support $\tau$-tilting modules, 
and that  the functor $\Res_{\Lambda}^{A}$ induces the inverse map between these posets.

In the paper, 
we also compare  silting complexes, $2$-term silting complexes and tilting complexes via the canonical adjoint pair  
and obtain similar poset isomorphisms between the posets  of these complexes.

The following theorem summarizes our results about the skew group algebra extension. 
\begin{theorem}\label{202404091842}
Let $\Lambda$ be a finite dimensional algebra with an action of a finite group $G$ 
and $A:= \Lambda * G$ the skew group algebra. 
Then the adjoint pair $(\Ind^{A}_{\Lambda},\Res_{\Lambda}^{A})$ induces the following 
commutative diagram of the poset morphisms 
\[
\begin{xymatrix}@C=80pt@R=15pt{
(\sttilt \Lambda)^{G} \ar@<2pt>[r]^{\Res^{A}_{\Lambda}} 
& (\sttilt A)^{\mod G} \ar@<2pt>[l]^{\Ind_{\Lambda}^{A}}  \\
(\2silt \Lambda)^{G} \ar@<2pt>[r]^{\Res^{A}_{\Lambda}}  \ar[u]^{\tuH^0} \ar@{^{(}->}[d]
& (\2silt A)^{\mod G} \ar@<2pt>[l]^{\Ind_{\Lambda}^{A}} \ar[u]_{\tuH^0} \ar@{^{(}->}[d] \\
(\silt \Lambda)^{G} \ar@<2pt>[r]^{\Res^{A}_{\Lambda}}   
& (\silt A)^{\mod G} \ar@<2pt>[l]^{\Ind_{\Lambda}^{A}}  \\
(\tilt \Lambda)^{G} \ar@<2pt>[r]^{\Res^{A}_{\Lambda}} \ar@{^{(}->}[u]  
& (\tilt A)^{\mod G} \ar@<2pt>[l]^{\Ind_{\Lambda}^{A}} \ar@{^{(}->}[u]  \\
}\end{xymatrix}
\]
where the horizontal arrows are poset isomorphisms inverse to each other.  
The vertical arrows $\tuH^{0}$ from the middle to the top are induce from the poset isomorphism $\tuH^{0}$ between 
$2$-term silting complexes and support $\tau$-tilting complexes. 
\end{theorem}

The top horizontal isomorphisms are established in Theorem \ref{202309301601}.  
The second, the third and the fourth  horizontal isomorphisms, as well as the commutativity of the second and third squares are established in Theorem \ref{202403121806}. 
Proposition \ref{202403121843} proves that two vertical arrows $\tuH^{0}$ are poset isomorphisms. 
The commutativity of the first square is easy to check and its verification is left to the readers.

In a subsequent work, we establish similar correspondences of torsion theories, semi-bricks and wide subcategories over $\Lambda$ and $A$.

\subsubsection{$\tau$-tilting finiteness and silting discreteness} 

In general \textup{(}$\mod G$\textup{)}-stability is difficult to check. 
However,   next theorem tells that under certain conditions, 
$\tau$-tilting finiteness (resp. silting discreteness) of $\Lambda$ is inherited to $A=\Lambda * G$ 
and   moreover that all support $\tau$-tilting modules (resp. silting complexes) over $A$ are \textup{(}$\mod G$\textup{)}-stable. 

The basic setup of the theorem is that,
as is often with many applications, 
an algebra $\Lambda = \kk Q/I$ is  given by a quiver $Q$ with relations $I$ 
and an action $G\curvearrowright \Lambda$ is  induced from an action $G\curvearrowright Q$ on the quiver $Q$. 
For a vertex $e\in Q_{0}$, we denote by $\stab(e) < G$ the stabilizer group.

\begin{theorem}\label{202404101701}
In the above setup, we further assume  that for each vertex $e \in Q_{0}$, the group algebra $\kk \stab(e)$ is local.
If $\Lambda$ is a $\tau$-tilting finite (resp. silting discrete), 
then so is $A$ and moreover all support $\tau$-tilting modules (resp. silting complexes) are \textup{(}$\mod G$\textup{)}-stable. 
Consequently the adjoint pair $(\Ind^{A}_{\Lambda}, \Res_{\Lambda}^{A})$ induces the following poset isomorphisms 
\[
\begin{split}
 \Ind_{\Lambda}^{A}: (\sttilt \Lambda)^{G}  \ \ &\longleftrightarrow \ \ \sttilt A : \Res_{\Lambda}^{A}\\
 \textup{(resp. }  \Ind_{\Lambda}^{A}: (\silt \Lambda)^{G} \ \ &\longleftrightarrow  \ \ \silt A : \Res_{\Lambda}^{A} \textup{)}\\
\end{split}
\]
\end{theorem}

Theorem \ref{202404101701} about  support $\tau$-tilting modules (resp. silting complexes) is proved as  
Theorem \ref{202309301357} (resp. Theorem \ref{202403311841}) under a slightly more general  condition.

We note that if the action of $G$ on the vertex set $Q_{0}$ is free, the assumption of Theorem \ref{202404101701} is satisfied. 
For the convenience of the readers, we recall the following well-known fact in modular representation theory of a finite group. 

\begin{lemma}[See for example {\cite[Corollary 3.2.4]{Schneider}}]\label{2023100319381}
Assume that $\kk$ is an algebraically closed field of characteristic $p=\chara \kk >0$. 
The group algebra $\kk G$ is local algebra if and only if $G$ is a  $p$-group. 
\end{lemma}

%
%

\subsection{Related results}\label{202406031540}
In this section, 
we  review recent related results which study a triangulated category $\cT$ 
with an action $G \curvearrowright \cT$ by a finite group $G$ and 
then we explain relationships between these results and our results. 
We assume that a triangulated category $\cT$ is $\Hom$-finite and idempotent complete in the sequel. 
We denote by $\cT^{G}$ the category of $G$-equivariant objects of $\cT$, which is triangulated, in nice situations, by \cite{Balmer, Chen: a note, Elagin}.

Chen-Chen-Ruan \cite{CCR} deals with a triangulated category $\cT$ equipped with an action $G\curvearrowright \cT$ by a finite abelian group $G$. 
The point here is that the category $\cT^{G}$ of $G$-equivariant object comes equipped with an action 
$\hat{G}\curvearrowright \cT^{G}$ by the character group $\hat{G}$. 
Then  as one of their main results \cite[Theorem A]{CCR}, they proved that 
$G$-stable tilting objects of $\cT$ is bijectively corresponds to $\hat{G}$-stable tilting objects of $\cT^{G}$ under the assumption that 
the group algebra $\kk G$ is a direct product of $\kk$.

Dell-Heng-Licata \cite{DHL}, first study stability conditions of a triangulated category $\cT$ equipped with an action $\cC \curvearrowright \cT$ 
of a fusion category $\cC$ (for fusion categories we refer \cite{EGNO}).  
The point here is that the category $\cT^{G}$ of $G$-equivariant object comes equipped with an action 
$\mod G \curvearrowright \cT^{G}$ by the tensor category $\mod G$. 
We note that they need to  assume that $\chara \kk \nmid |G|$ so that the tensor category $\mod \kk G$ is a fusion category. 
One of their main result  \cite[Theorem B]{DHL} proves that 
the space $\textup{stab}_{G}(\cT)$ of $G$-equivariant stability conditions of $\cT$ is homeomorphic to 
the space $\textup{stab}_{\mod G}(\cT^{G})$ of ($\mod G$)-equivariant stability conditions of $\cT$. 
We note that 
in \cite[Theorem B]{DHL}, the action $G\curvearrowright \cT$ is reinterpreted to the action $\mathsf{vec}(G) \curvearrowright \cT$ 
by the category $\mathsf{vec}(G)$ of $G$-graded vector spaces.  
(The category $\mathsf{vec}(G)$ is isomorphic to the category $\kk G^{*}\mod$ of $\kk G^{*}$-modules where 
$\kk G^{*}=\Hom_{\kk}(\kk G, \kk)$ denotes the dual Hopf algebra of $\kk G$. 
 The above reinterpretation is explained in Appendix \ref{202403292142}.)

Now we explain relationships between above results and our results. 
Thus as in the previous section, let $\Lambda$ be a finite dimensional algebra with an action $G\curvearrowright \Lambda$ by a finite group 
and $A:= \Lambda *G$.
The category $\mod \Lambda$ of $\Lambda$-modules is equipped with the induced action $G\curvearrowright \mod \Lambda$. 
It is well-know that $(\mod \Lambda)^{G}\simeq \mod A$. 
Now we assume that $\chara \kk \nmid |G|$. 
Then by\cite[Proposition 4.5]{Chen: a note},  we have an equivalence 
$\sfD^{\mrb}(\mod \Lambda)^{G}\simeq \sfD^{\mrb}( (\mod \Lambda)^{G}) \simeq \sfD^{\mrb}(\mod A)$ of triangulated categories. 
As we mentioned before, silting complexes and stablity conditions are closely related. 
Therefore, 
since 
our poset isomorphism $(\silt \Lambda)^{G} \cong (\silt A)^{\mod G}$ of silting complexes is verified without the assumption $\chara \kk \nmid |G|$, 
we conjecture that there exists a homeomorphism 
$\textup{stab}_{G}(\sfD^{\mrb}(\mod \Lambda)) \cong \textup{stab}_{\mod G}(\sfD^{\mrb}(\mod A))$ without any assumption on $\chara \kk$. 
 This conjecture and its variant for algebraic varieties are left to our future study.

If we  further assume that the  group algebra $\kk G$ is a direct product of $\kk$ as in \cite{CCR}, 
then \textup{(}$\mod G$\textup{)}-stable objects of $\sfD^{\mrb}(\mod A)$ are precisely $\hat{G}$-stable objects of $\sfD^{\mrb}(\mod A)$ 
(see for the discussion after Definition \ref{202309271829}). 
Thus our poset isomorphism $(\tilt \Lambda)^{G} \cong (\tilt A)^{\mod G}$ established for a finite group $G$ 
which is not necessarily abelian, can be looked as 
a partial generalization of the result given in  \cite{CCR}
 in  the situation that $\cT = \sfD^{\mrb}(\mod \Lambda)$ and an action $G\curvearrowright \cT$ 
is induced from an action $G \curvearrowright \Lambda$ on an algebra $\Lambda$.

\subsubsection{}

Homma \cite{Homma} deals with the situation that
a (not necessarily finite) group $G$ acting on a $\Hom$-finite triangulated  category $\cT$.
Compare to the setup of our paper, 
 a $\Hom$-finite $\kk$-linear category $\cT$ plays the role of the derived category $\Dbmod{\Lambda}$ with the induced $G$-action. 
Using covering theory, he studied the poset of $G$-stable silting objects of $\cT$. 
One of his main results shows that if  an action  of $G$ on $\ind \cT$ is free, 
then the poset of $G$-stable silting objects of $\cT$ is isomorphic as posets to the silting poset of $\cT'$ 
where $\cT'$ is a $G$-precovering of $\cT$ such that the thick-hull of the image of the induction functor coincide with $\cT'$. 
In the case $\cT=\Dbmod{\Lambda}$, 
the consequence of this result may coincide with the silting part of Theorem  \ref{202404101701}. 
 However, the assumption of freeness for the action $G\curvearrowright \ind \Dbmod{\Lambda}$ is  stronger than our assumption.

\subsection{Organization of the paper}

Section \ref{202404102219} gives preliminaries about a skew group algebra extension $\Lambda \subset A$ where $A:= \Lambda *G$. 
We recall well-known facts about a skew group algebra extension. 
In the case $\chara \kk \nmid |G|$, it is well-known that the Auslander-Reiten translations $\tau$ (of $\Lambda$ and of $A$) commute 
with the functors $\Ind_{\Lambda}^{A}, \ \Res_{\Lambda}^{A}$. 
As one of the technical heart of the paper, we discuss this commutativity without the assumption on the characteristic of the base field.
Other important notions introduced in this section are 
\textup{(}$\mod G$\textup{)}-action on modules $N$ over $A$ and stability with respect to this action which is called \textup{(}$\mod G$\textup{)}-stability.

In Section \ref{202404102224} and Section \ref{202404102225}, 
we study $\tau$-tilting theory and silting theory of a skew group algebra extension, respectively 
and prove the corresponding part of Theorem \ref{202404091842} and Theorem \ref{202404101701}.

In Section \ref{202404102228}, applying results obtained in the previous sections, 
we study silting theory and $\tau$-tilting theory of preprojective algebras and folded mesh algebras.

 In Appendix \ref{202404011608}, we explain that our results about a skew group algebra extensions 
 can be reinterpreted to that about a $G$-graded algebras $A=\bigoplus_{g\in G} A_{g}$ and 
 the canonical adjoint pair $V: \grmod A \rightleftarrows \mod A: U$ between 
 the categories of graded  $A$-modules and that of  ungraded $A$-modules, via Cohen-Montgomery duality. 
 We also explain that 
 specializing the results about a $G$-graded algebra, 
 we obtain results about an extension $A_{e_{G}} \subset A$ of a strongly  $G$-graded algebra by 
 its degree $e_{G}$-part $A_{e_{G}}$. 
 By  this reinterpretation,  our results  gives generalizations of previous papers 
 \cite{BMM, KK}.
 
 In Appendix \ref{202403292142}, we discuss a formal aspect from  theory of tensor categories and Hopf algebras. 
 We provide a formalism  of our result that allows us to 
 give a  possible generalization of our result to an algebra $\Lambda$ with an action of a Hopf algebra $H$.

\subsection{Conventions}

In the paper 
$\kk$ denotes a  (not necessarily algebraically closed) field   of arbitrary characteristic.
Undecorated tensor product $\otimes$ always denotes the tensor product $\otimes_{\kk}$ over $\kk$.  

Groups are assumed to be finite groups. 
Algebras are assumed to be a  finite dimensional $\kk$-algebra. 
Modules over an algebra $\Lambda$ are always finite dimensional left $\Lambda$-modules, unless otherwise is mentioned.  
An action $G \curvearrowright \Lambda$ of a group $G$ on an algebra $\Lambda$ is assumed to 
preserve the algebra structure of $\Lambda$.

Let $M$ be a $\Lambda$-module. 
We denote its basic part  by $M^{\basic}$, i.e., a basic module  such that $\add M^{\basic} =\add M$. 
We note that $M^{\basic}$ is unique up to isomorphisms. 
We denote by $|M|$ the number of the indecomposable direct summands of $M^{\basic}$. 

In the same way, we define the basic part $M^{\basic}$ of an object $M$ of the derived category $\Dbmod{\Lambda}$.

\subsection*{Acknowledgment}
The authors thank H. Asashiba for answering a question about relationship between skew group algebras and orbit categories,  
and suggesting  his paper \cite{Asashiba}. 
We also appreciate his   comments and suggestions to the first draft of this paper.
The authors thank T. Adachi for discussions about the paper \cite{Adachi-Kase} by him with R. Kase. 
The authors thank T. Homma for explaining his results give in \cite{Homma}. 
The authors thank N. Hiramae for pointing out an error in the first draft of this paper. 
The authors would like to express their gratitude to R. Wang for bringing the papers \cite{CW1,CW2} to their attention.

The first author is partly supported by Grant-in-Aid for JSPS Fellows JP22J01155.
The fourth author is partly supported by  Grant-in-Aid for Scientific Research (C) Grant Number JP21K03210. 
The fifth author  is partly supported by Grant-in-Aid for Scientific Research (C) Grant Number JP20K03539.

\section{Preliminaries about skew group algebra extensions}\label{202404102219}

Let $\Lambda$ be an algebra with an action of a group  $G$. 
Recall that the \emph{skew group algebra}  $A := \Lambda \ast G$ is defines in the following way: 
The underlying $\kk$-vector space of $A$ is defined to be $\Lambda \otimes \kk G$ where $\kk G$ is the group algebra of $G$. 
For $r \in \Lambda$ and $g \in G$, we set $r* g := r \otimes g \in \Lambda \otimes \kk G$.
Then the multiplication of $A$ is given by 
\[
(r*g)(s*h) := rg(s) *gh \ \ \ (r,s \in \Lambda, \ g,h\in G).
\]

\subsection{$G$-action on $\Lambda$-modules and \textup{(}$\mod G$\textup{)}-action on $A$-modules}\label{202403291602}

\subsubsection{$G$-action and $G$-stability}\label{202403291718}
For $g \in G$ and a $\Lambda$-module $M$, we denote by ${}_{g}M$ the $\Lambda$-module whose underlying vector space 
is that of $M$ and the action of $\Lambda$ is twisted by $g$, i.e. 
\[
r \cdot_{{}_{g}M} m := g(r)\cdot_{M} m
\]
for $r \in \Lambda, g \in G$ and $m\in M$. 
We remark that ${}_{h}({}_{g}M) \cong {}_{gh}M$. 
Note that the map ${}_{g}\Lambda \to \Lambda, r \mapsto g^{-1}(r)$ is an isomorphism of $\Lambda$-modules.

\begin{definition}\label{202309151111}
A $\Lambda$-module $M$ is called $G$-\emph{stable} if ${}_{g}M \cong M$ for all $ g\in G$. 
\end{definition}

\subsubsection{$\mod G$ action and \textup{(}$\mod G$\textup{)}-stability} 

\begin{definition}\label{202309161735}
Let $N$ be an $A=\Lambda * G$-module and $X$ a $G$-module. 
We define an $A$-module $X \otimes^{G}_{\kk} N$ in the following way. 
As a $\kk$-vector space $X \otimes^{G}_{\kk} N$ is defined to be $X \otimes_{\kk} N$. 
We define an action of an element  $r* g \in A$ to $x \otimes n \in X \otimes_{\kk} N$ by the following formula:
\[
(r* g)\cdot (x \otimes n) := g(x) \otimes  (r* g) n.
\]
\end{definition}

We note that this construction is already used in representation theory of skew group algebras 
(see e.g., \cite[Definition 5.17]{LW}).

\begin{remark}\label{202407021157}
The construction $(X,M) \mapsto X\otimes_{\kk}^{G} M$ given in Definition \ref{202309161735} 
is natural in $X \in \mod G$ and $M \in \mod \Lambda$ respectively and hence induces a functor $\otimes_{\kk}^{G}: (\mod G) \times (\mod A) \to \mod A$. 
It is straightforward to check that this functor together with the isomorphisms given in Lemma \ref{2023010021729} below endows  
the category $\mod A$ with a structure of a left module category over the tensor category $\mod G$ in the sense of \cite[Definition 7.1.2]{EGNO}.
(It should be noted that familiarity with tensor categories is not necessary for understanding this paper, with the exception of Appendix B.)

In other words,  the construction of Definition \ref{202309161735} gives an action  $\mod G \curvearrowright \mod A$. 
We  note that as is explained in Section \ref{202404102225}, 
this action $\mod G \curvearrowright \mod A$ on $\mod A$ extends to  an action $\mod G \curvearrowright \Dbmod{A}$  on   the derived categories $\sfD^{\mrb}(\mod A)$, which coincides with the \textup{(}$\mod G$\textup{)}-action used in \cite{DHL}.
\end{remark}

\begin{example}
If $\Lambda=\kk$ with the trivial $G$-action, then the skew group algebra $A$ is the group algebra $\kk G$ and 
the $A$-module  $X \otimes_{\kk}^{G} N$ is the $\kk G$-modules $X \otimes_{\kk} N$ with the diagonal action.
\end{example}

We leave the verification of the following lemma to the readers.

\begin{lemma}\label{2023010021729}
For an $A$-module $N$ and $G$-modules $X, Y$, the following assertions hold.
\begin{enumerate}[(1)]
\item $\kk_{G} \otimes_{\kk}^{G} N \cong N$. 

\item $(X \otimes_{\kk}^{G} Y) \otimes_{\kk}^{G} N \cong X \otimes_{\kk}^{G} (Y \otimes_{\kk}^{G} N)$. 
\end{enumerate}
\end{lemma} 

Thanks to the second statement of above lemma, 
we may write both of $(X \otimes_{\kk}^{G} Y) \otimes_{\kk}^{G} N $ and $ X \otimes_{\kk}^{G} (Y \otimes_{\kk}^{G} N)$ 
by $X\otimes_{\kk}^{G} Y \otimes_{\kk}^{G} N$.

\begin{example}\label{202309271819}
Observe that $\Lambda$ has a canonical $A$-module structure. 
For a $G$-module $X$, we often write $\Lambda * X = X \otimes_{\kk}^{G} \Lambda$ 
and in that case an element $x \otimes r$ is written as $r *x$. 
The $A$-module $\Lambda * \kk G = \kk G\otimes_{\kk}^{G} \Lambda$ is canonically isomorphic to the $A$-module 
$A = \Lambda * G$. 

Let $\kk G =\bigoplus_{i=1}^{n} P_{i}$ be a decomposition as a $G$-module. 
Then we have a decomposition $A =\bigoplus_{i=1}^{n} \Lambda * P_{i}$ as an $A$-module. 
\end{example}

\begin{lemma}\label{202309150826}
Let $X$ be a $G$-module. 
Then $X \otimes_{\kk}^{G} A$ is isomorphic to $X \otimes_{\kk} A \cong A^{\oplus \dim X}$ as $A$-$\Lambda$-bimodules.
\end{lemma} 

\begin{proof}
It is straightforward to check that the following map is an isomorphism of $A$-$\Lambda$-bimodules 
\[
\phi: X \otimes_{\kk}^{G} A \to X \otimes_{\kk} A, \phi(x \otimes (r *g) ) := g^{-1}(x) \otimes  (r * g).
\]
\end{proof}

\begin{corollary}\label{202310021728}
Let $Q$ be a projective $A$-module and $X$ a finite dimensional $G$-module. 
Then $X\otimes_{\kk}^{G}Q$ is a projective $A$-module.
\end{corollary}

%

Now we introduce the notion of \textup{(}$\mod G$\textup{)}-stability.

\begin{definition}\label{202309171251}
An $A$-module $N$ is called ($\mod G$)-\emph{stable} 
if $X \otimes_{\kk}^{G} N \in \add N$ for any $X \in \mod G$. 
\end{definition}

We introduce $\XX$-stability.  

\begin{definition}[{The $\mathbb{X}$-stability(cf.\cite[p.1070]{ZH})}]\label{202309271829}
An $A$-module $N$ is called $\XX$-\emph{stable} 
if $S \otimes_{\kk}^{G} N \in \add N$ for any $1$-dimensional simple $G$-module $S$. 
\end{definition}

We explain that the above definition of $\XX$-stability coincides with the original definition given by Zhang-Huang \cite{ZH}.
Let $\chi: G \to \kk^{\times}$ be a group homomorphism and $\kk_{\chi}$ the induced $G$-module. 
For a $1$-dimensional $G$-module $S$, there exists a unique $\chi: G \to \kk^{\times}$ 
such that $S\cong \kk_{\chi}$. 
Then $S \otimes_{\kk}^{G}N = \kk_{\chi}\otimes^{G}_{\kk} N$ is ${}^{\chi^{-1}}N$ in the notation of \cite{ZH}. 
Observe that 
in case of $N$ is basic, then we have $S\otimes_{\kk}^{G} N \in \add N$ if and only if $\kk_{\chi}\otimes_{\kk}^{G} N \cong N$. 
Thus, we see that for a basic $A$-module $N$, the two $\XX$-stability are the same notion.

We note that  an $A$-module $N$ is \textup{(}$\mod G$\textup{)}-stable, then it is $\mathbb{X}$-stable. 
In Example \ref{202309271835}, we give an example of support $\tau$-tilting $A$-modules which is $\mathbb{X}$-stable, but not \textup{(}$\mod G$\textup{)}-stable.

\subsubsection{}

The adjoint property of \textup{(}$\mod G$\textup{)}-action given in the following proposition plays a key role in the proof of Theorem \ref{202309301357}.
\begin{proposition}\label{2023010031837}
Let $X$ be a finite dimensional $G$-module and $\tuD(X)=\Hom_{\kk}(X,\kk)$ the dual $G$-module of $X$. 
Then the pair $(X\otimes_{\kk}^{G}-, \tuD(X) \otimes_{\kk}^{G} -)$ of endofunctors of $\mod A$ is an adjoint pair.
\end{proposition}

This proposition follows from a general result in the theory of module categories over a tensor category \cite[Proposition 7.1.6]{EGNO}. 
But for the convenience of the readers, we provide a direct proof.

\begin{proof}
Let $\{x_{i}\}_{i=1}^{n}$ be a basis of $X$ and $\{x^{i} \}_{i=1}^{n}$ the dual basis of $\tuD(X)$. 
For an $A$-module $N$, we define maps 
$\epsilon_{N}: N \to \tuD(X) \otimes_{\kk}^{G} X \otimes_{\kk}^{G} N$ and 
$\eta_{N}: X \otimes_{\kk}^{G} \tuD(X) \otimes_{\kk}^{G} N \to N$ to be 
\[
\begin{split}
\epsilon_{N} & : N \to \tuD(X) \otimes_{\kk}^{G} X \otimes_{\kk}^{G} N, \ \epsilon_{N}(n) := \sum_{i =1}^{n} x^{i} \otimes x_{i} \otimes n \\
\eta_{N} &: X \otimes_{\kk}^{G} \tuD(X) \otimes_{\kk}^{G} N \to N, \ \eta_{N}(x \otimes \xi \otimes n ) := \xi(x) n. 
\end{split}
\]
Then, it is straightforward to check that these maps are homomorphisms of $A$-modules natural in $N$ 
that satisfy the triangle identities $(\eta_{X\otimes^{G} N})\circ (X\otimes^{G} \epsilon_{N}) = \id_{X\otimes^{G} N}$, 
$(\tuD(X)\otimes^{G} \eta_{N})\circ (\epsilon_{\tuD(X)\otimes^{G} N}) = \id_{\tuD(X)\otimes^{G} N}$. 
\end{proof}

%
%
%
%
%
%
%
%
%
%
%
%
%
%

\subsection{Induction-Reduction adjunction}\label{202403291603}

Recall that  the extension (i.e., the canonical inclusion) $\Lambda \to A$ is a Frobenius extension, 
that means that,  
 $A$ is a free as left $\Lambda$-module  and that 
we have an isomorphism of $A$-$\Lambda$-bimodules 
\[
\phi: A \to \Hom_{\Lambda}(A,\Lambda),
\]
which is given by $\phi(r * g)(s *h) := sh(r) \delta_{g,h^{-1}}$. 
Consequently the functors $\Ind_{\Lambda}^{A} =A\otimes_{\Lambda}- , \ \Hom_{\Lambda}(A, -): \mod \Lambda \to \mod A$ 
are naturally isomorphic.
Hence the induction functor $F:= \Ind_{\Lambda}^{A}$ is left and right adjoint of the restriction functor $R: = \Res_{\Lambda}^{A}$. 
\begin{equation}\label{202309151217}
\begin{xymatrix}{
 \mod A \ar[d]^{R} \\
 \mod \Lambda \ar@/^20pt/[u]^{F } \ar@/_20pt/[u]_{F}
}\end{xymatrix}
\end{equation}

\subsubsection{Basic isomorphisms}

We collect isomorphisms of $\Lambda$-modules and $A$-modules, that play important roles in the sequel. 
Although these are well-known and can be (partly) found in the standard references (e.g., \cite{RR}), 
we provides proofs for the convenience of the readers.

\begin{lemma}\label{202309151023}
\begin{enumerate}[(1)]
\item 
For a $\Lambda$-module $M$, an element $g \in G$ and $X \in \mod G$, we have the following isomorphisms 
\[
\begin{split}
&\textup{(a)} \ F({}_{g}M)  \cong F(M),  \ \textup{(b)} \  RF(M) \cong  \bigoplus_{g \in G} {}_{g} M, \\
 & \textup{(c)} \  FRF(M) \cong F(M)^{\oplus |G|}, \ \textup{(d)} \ X \otimes_{\kk}^{G} FM \cong (FM)^{\oplus \dim X}.
\end{split} 
\]

\item 
For an $A$-module $N$, an element $g\in G$ and $X \in \mod G$, we have the following isomorphisms. 
\[
\begin{split}
&\textup{(a)} \  R(N) \cong {}_{g}R(N), \ \textup{(b)} \  FR(N)\cong A\otimes_{\Lambda}N \cong \kk G\otimes_{\kk}^G N, \\
&\textup{(c)} \ RFR(N) \cong R(N)^{\oplus |G|}, \ \textup{(d)} \ R(X\otimes^{G}_{\kk} N) \cong X \otimes_{\kk} N \cong N^{\oplus \dim X}. 
\end{split}
\]
\end{enumerate}
\end{lemma}

\begin{proof}
(1) 
The isomorphism (a)   is given by 
\[
\phi: A \otimes_{\Lambda} {}_{g}M \to A \otimes_{\Lambda} M, \ \phi((r* h) \otimes m ) := (r * hg^{-1}) \otimes m. 
\]
The isomorphism (b)  is given by 
\[
\phi: \bigoplus_{g\in G} {}_{g}M \to A\otimes_{\Lambda} M , \phi( (m_{g})_{g \in G}) := \sum_{g \in G} (1* g^{-1}) \otimes m_{g}.
\]
The isomorphism (c) follows from (a) and (b). The isomorphism (d) is a consequence of Lemma \ref{202309150826}.

(2) It is straightforward to check that the map  $\phi_{N}: R(N) \to {}_{g}R(N) , \  \phi(n) := (1 * g) n$ 
gives the desired  isomorphism (a).

It follows from  the definition of $F, R$ that $RF(N) \cong A\otimes_{\Lambda} N$. 
It is straightforward to prove that the  map below is an isomorphism of $A$-modules.
\[
\phi: \kk G \otimes_{\kk}^{G} N \to A \otimes_{\Lambda} N = FR(N), \ \phi(g \otimes n) := (1 * g) \otimes  (1*g^{-1}) n
\]

The isomorphism (c) follows from  (a) and (1-b). 
It is straightforward to check that the map $\phi: R(X\otimes_{\kk}^{G} N) \to X \otimes_{\kk} R(N), \ \phi(x \otimes n) := x \otimes n$ 
is an isomorphism of $\Lambda$-modules. 
\end{proof}

We point out the following immediate consequence of Lemma \ref{202309151023}.

\begin{corollary}\label{202309151718}
\begin{enumerate}[(1)]
\item For a $\Lambda$-module $M$, the $A$-modules $F(M)$ is \textup{(}$\mod G$\textup{)}-stable.
Moreover we have $\add FRF(M) = \add F(M)$.

\item For an $A$-module $N$, the $\Lambda$-modules $R(M)$ is $G$-stable. 
Moreover we have $\add RFR(N) = \add R(N)$. 
\end{enumerate}
\end{corollary}

%
%
%
%
%
%
%
%
%
%
%
%
%
%
%
%
%
%
%
%
%
%
%
%
%
%
%
%

\subsubsection{A criterion of $G$-stability}

As a consequence of Lemma \ref{202309151023}, we have the following criterion of $G$-stability
 
\begin{corollary}\label{202310021718}
A $\Lambda$-module $M$ is $G$-stable if and only if  we have $\add M = \add RF(M)$. 
\end{corollary}

\subsubsection{A criterion of \textup{(}$\mod G$\textup{)}-stability}

Recall that an $A$-module $N$ is called \emph{rigid} if $\Ext_{A}^{1}(N,N) = 0$. 
To obtain a criterion of \textup{(}$\mod G$\textup{)}-stability similar to Corollary \ref{202310021718}, 
we need to assume that an $A$-module $N$ to be rigid.

\begin{proposition}[{cf. \cite[Lemma 3.7]{KK}}]\label{202309150913}
For a rigid $A$-module $N$, the following conditions are equivalent. 

\begin{enumerate}[(1)]

\item $N$ is \textup{(}$\mod G$\textup{)}-stable.

\item $S \otimes_{\kk}^{G} N \in \add N$ for any simple $G$-modules $S$.

\item $\add FR(N) = \add N$.

\end{enumerate}
Moreover, if we further assume  one of the following conditions: (i) $\chara \kk \nmid |G|$, (ii) $\kk G$ is local, 
then the above conditions are equivalent to the following condition (4). 

\hspace{-16pt} (4) $FR(N) \in \add N$. 

\end{proposition}

We leave the verification of the following lemma to the readers, since it can be prove in the same way of \cite[Lemma 3.6]{KK}. 

\begin{lemma}[{cf. \cite[Lemma 3.6]{KK}}]\label{202309150909}
Let $N$ be an $A$-module and $X$ a $G$-module. 
Assume that $N$ is rigid and that for any composition factors $S$ of $X$, we have $S \otimes^{G}_{\kk } N \in \add N$. 
Then  we have $X \otimes^{G}_{\kk} N = \bigoplus S \otimes_{\kk}^{G} N$ where $S$ runs over all composition factor of $X$.  
\end{lemma}

\begin{proof}[Proof of Proposition \ref{202309150913}]
(3)$\Rightarrow$(1) follows from Lemma \ref{202309151023} (1-d). 
(1)$\Rightarrow$(2) is trivial. 
We prove the implication (2)$\Rightarrow$(3). 
By  Lemma \ref{202309151023}(2-b) and Lemma \ref{202309150909}, 
we have 
\[
A \otimes_{\Lambda} N \cong \kk G \otimes_{\kk}^{G} N \cong \bigoplus S\otimes_{\kk}^{G} N
\]
where $S$ runs all composition factors of $\kk G$. 
Since we are assuming (2) holds, it follows that  $A \otimes_{\Lambda} N \in \add N$. 
Since the trivial $G$-module $\kk_{G}$ is a composition factor of $\kk G$, we see that $N \cong \kk_{G}\otimes_{\kk}^{G} N $ belongs to $\add(A\otimes_{\Lambda} N)$. 

It is clear (3) always implies (4). 
By Lemma \ref{202404031821} below, we see that (4) together with one of (i), (ii) implies (3).
\end{proof}

\begin{lemma}\label{202404031821}
Let $N$ be an $A$-module. If one of the following conditions holds, 
then we have $N \in \add FR(N)$. 

\begin{enumerate}[(i)]

\item $\chara \kk \nmid |G|$. 

\item $N$ is rigid and $\kk G$ is local. 
\end{enumerate}
\end{lemma} 

\begin{proof}
It is enough to show that the canonical map $f: FR(N) \to N$ splits. 
We note that  we have $f= \epsilon \otimes^{G}_{\kk}N :  \kk G \otimes_{\kk}^{G} N \to \kk_{G}\otimes_{\kk}^{G} N$ 
where  $\epsilon: \kk G \to \kk_{G}$ denotes the augmentation  map. 

If the condition (i) holds. Then the canonical map $g$ splits and hence so does $f$. 

If the condition (ii) holds. Then every $G$-module $X$ has a surjective homomorphism $X \to \kk_{G}$ of $G$-modules. 
By induction of length of $X$ and utilizing the rigidity of $N$, we can show  that the induced map $X \otimes_{\kk}^{G} N \to \kk_{G}\otimes_{\kk}^{G} N$ splits and $X\otimes^{G}_{\kk} N \in \add N$. 
\end{proof}

\begin{remark}
Even if an $A$-module $N$ is not rigid, the implications (3) $\Rightarrow$ (1) $\Rightarrow$ (2) hold true, 
but converses do not hold true in general. 

We provide an example that (2) $\Rightarrow$ (1) $\Rightarrow$ (3) does not hold true in general. 
Assume that $\chara \kk = 2$. Let $\Lambda=\kk$ be the base field with the trivial action by a cyclic group $G=C_{2}$ of order $2$. 
Then $A= \kk G$ is a usual group algebra. 
Since there is only one simple $G$-module $\kk_{G}$, the trivial $G$-module, any $A$-module $N$ satisfies the condition (2). 
But the $A$-module $N=\kk_{G}$ is not \textup{(}$\mod G$\textup{)}-stable, since $\kk G \otimes_{\kk}^{G} \kk_{G} = \kk G \notin \add \kk_{G}$. 
It is easy to see that  $N' := \kk_{G} \oplus \kk G$ is \textup{(}$\mod G$\textup{)}-stable but $\add FR(N') \subsetneq \add N'$.  
\end{remark}

\subsubsection{The local case}

Assume that  the group algebra $\kk G$ is local.  
Then the the unique simple module is the trivial representation $\kk_{G}$ of $G$. 
In this case, the condition (2) of Proposition \ref{202309150913} is always satisfied.
Thus as a consequence, we have the following corollary. 

\begin{corollary}\label{202310131118}
If $G$ is a finite group such that $\kk G$ is local, 
then every rigid module $N$ over $A=\Lambda * G$ is \textup{(}$\mod G$\textup{)}-stable as well as it satisfies $\add N = \add FR(N)$.
\end{corollary}

%

\subsubsection{A lemma}

The following lemma is used to prove Lemma \ref{202309281928}. 

\begin{lemma}\label{202309151607} 
Let $e, e' \in \Lambda$ be  idempotent elements, $r \in e\Lambda e'$  and $g \in G$. 
Then the following assertions hold.

\begin{enumerate}[(1)]

\item We have an isomorphism ${}_{g} \Lambda e \cong \Lambda g^{-1}(e)$ of $\Lambda$-modules.

\item We have an isomorphism $RF(\Lambda e) \cong \bigoplus_{g \in G} \Lambda g(e)$ of $\Lambda$-modules. 

\item 
We denote the right multiplication map $r: \Lambda e \to \Lambda e'$ by the same symbol. 
Then under the above isomorphism, we have $RF(r) = \bigoplus_{g \in G}g(r)$.
\[
RF(r) : RF(\Lambda e) \cong \bigoplus_{g\in G} \Lambda g(e) \xrightarrow{ \bigoplus g(r) } \bigoplus_{g \in G} \Lambda g(e') =RF(\Lambda e')
\]
\end{enumerate}

\end{lemma}

\begin{proof}
(1)
The desired  isomorphism is given by 
\[
\phi: {}_{g} \Lambda e \to \Lambda g^{-1}(e), \ \phi(re) := g^{-1}(r)g^{-1}(e).
\]
(2)(3) follow from (1) and Lemma  \ref{202309151023}(1-b).
\end{proof}

%
%
%
%
%
%
%
%
%
%
%
%
%
%
%
%
%

\subsection{Compatibility of the Auslander-Reiten translation $\tau$ with  $F$ and $R$}\label{202309171607}

\subsubsection{Recollection of the radical morphisms}

For $M, M' \in \mod \Lambda$, the \emph{radical} $\rad(M',M)$ is defined to be 
a subspace of $\Hom_{\Lambda}(M',M)$ consisting of such elements $f:M' \to M$ that satisfy the following property: 
for any $L\in \ind \Lambda$ and any morphisms $s:L  \to M', t:M \to L$, 
the composition $tfs: L \to L$ is not an isomorphism.  
We note that $\rad(-,+)$ is a $\kk$-linear sub-bifunctor of $\Hom_{\Lambda}(-,+)$. 
We also note that $\rad(M,M)$ is the Jacobson radical of $\End_{\Lambda}(M)$ \cite[Proposition A.3.5]{ASS}.

 Recall that a homomorphism $g: M \to M''$ is called \emph{right minimal} if 
 any endomorphism $h: M \to M $ such that $g\circ h = g$ is an isomorphism. 
 
 We need the following two lemmas about minimal morphisms and radical morphisms. 
 
 \begin{lemma}\label{202309281150}
 Let $0 \to M'\xrightarrow{ f}  M \xrightarrow{ g} M''$ be an exact sequence in $\mod \Lambda$. 
 Then $g$ is right minimal if and only if $f \in \rad(M',M)$.
 \end{lemma}
 
 \begin{proof}
 ``If'' part. Assume that $f\in \rad(M',M)$. 
 Let $h: M \to M$ be such that $g\circ h = g$. Then, $g\circ (h- \id_{M})= 0$. 
 It follows that there exists $k: M \to M'$ such that $h - \id_{M}= f\circ k$. 
 Since $h -\id_{M} = f\circ k$ belongs to $\rad(M,M)$, it is nilpotent.  
 Therefore, $h =\id_{M}+(h -\id_{M})$ is an automorphism of $M$. 
 
 ``Only if'' part. We prove that if $f$ does not belongs to $\rad(M',M)$, then $g$ is not a right minimal. 
 Let $L$ be an indecmoposable $\Lambda$-module and $s:L \to M', t:M \to L$ homomorphisms of $\Lambda$-modules
such that $t\circ f\circ s$ is an automorphism of $L$. 
We set $l:=t\circ f \circ s$ and $h := \id_{M}- f\circ s\circ l^{-1}\circ t$. 
Then, since $g\circ f=0$, we have $g\circ h = g$. 
On the other hand, $t\circ h = t - t\circ f \circ s \circ l^{-1} \circ t= t-t=0$. 
Since $t \neq 0$, we conclude that $h$ is not an automorphism. 
 \end{proof}
 
 \begin{lemma}\label{202309281219}
 Let $M, M'$ be  $\Lambda$-modules,  $P$  projective $\Lambda$-modules 
 and 
 $f:M \to P$, $g: M' \to M$   homomorphisms of $\Lambda$-modules. 
 Assume that $g: M' \to M$ is surjective. 
 Then $f $ belongs to $\rad(M,P)$ if and only if $f\circ g$ belongs to $\rad(M',P)$. 
 \end{lemma} 
 \begin{proof}
 ``Only if'' part is clear. 
 We prove ``if'' part by contraposition. 
 Assume that  $f$ does not belong to $\rad(M,P)$. 
Then there exists an indecomposable $\Lambda$-module $L$ and homomorphisms $s: L \to M, \ t: P \to L$ such that
 $l:=t\circ f \circ s$ is an automorphism of $L$. 
 It follows that $f\circ s \circ l^{-1}$ is a section of $t$ (i.e., $t\circ( f\circ s \circ l^{-1}) = \id_{L}$) and 
 hence the $\Lambda$-module $L$ is a projective.  
 Since $g: M' \to M$ is surjective, there exists a homomorphism $r: L \to M'$ such that $g\circ r = s$. 
  It follows that $t \circ (f\circ g ) \circ r = t\circ f \circ s$ is  an automorphism. 
  This shows that $f\circ g$ does not belong to $\rad(M',P)$.  
  \[
  \begin{xymatrix}{
M' \ar[rr]^{g} && M \ar[rr]^{f} && P \ar[dl]^{t} \\
&&& L \ar@{-->}[lllu]^{r} \ar[ul]^{s}
}\end{xymatrix}
\]
\end{proof}

\subsubsection{The Jacobson radicals and radical morphisms}

The following lemma  about the Jacobson radicals seems to be well-known. 
But for the convenience of the readers we provide a proof.

\begin{lemma}\label{202309151536}
Let $J_{\Lambda}$ and $J_{A}$ be the Jacobson radicals of $\Lambda$ and $A$. 
Then the following assertions hold.

\begin{enumerate}[(1)]
\item For an element $g\in G$, we have $g(J_{\Lambda}) = J_{\Lambda}$.
\item $J_{\Lambda} \subset J_{A}$.
\end{enumerate}
\end{lemma}

\begin{proof}
(1)
Let $\frkm$ be a maximal left ideal of $\Lambda$. 
Then the image $g(\frkm)$ by the action of $g \in G$ is again a maximal left ideal of $\Lambda$. 
It follows that $g(J_{\Lambda}) = J_{\Lambda}$. 

(2)
Let $j \in J_{\Lambda}$. To prove $j \in J_{A}$, we show that for any $\alpha, \beta \in A$, the element $1-\alpha j \beta$ of $A$ is invertible.
We set $\gamma := \alpha j \beta$. 
If we write $\alpha = \sum_{g\in G} \alpha_{g} * g, \ \beta=\sum_{g\in G}\beta_{g}* g$, 
then $\gamma = \sum_{g,h} \alpha_{g} g(j)g(\beta_{h}) * gh$. 
It follows from (1) and the fact that $J_{\Lambda}$ is a two-sided ideal that 
$\gamma$ is of the form $\sum_{g}\gamma_{g} * g$ for some $\gamma_{g} \in J_{\Lambda}$. 
Since $\Lambda$ is finite dimensional,  the Jacobson radical $J_{\Lambda}$ is nilpotent. 
Hence $\gamma$ is nilpotent. It follows that $1-\gamma = 1 -\alpha j \beta$ is invertible. 
\end{proof}

We fix a complete set $\{ e_{i} \}_{i\in I }$ of primitive orthogonal idempotent element of $\Lambda$. 
Let $P, P'$ be projective $\Lambda$-modules.   
We take decomposition 
$P=\bigoplus_{s=1}^{m}P_{s}, \ P' =\bigoplus_{t=1}^{n} P'_{t}$ into indecomposable  projective $\Lambda$-modules 
$P_{s}\cong \Lambda e_{i(s)}, P'_{t} \cong \Lambda e_{j(t)}$.  
Recall that a homomorphism $f: P \to P'$ of $\Lambda$-modules belongs to $\rad(P,P')$  if and only if 
each component $f_{ts}: P_{s} \to P'_{t}$ belongs to the radical $e_{i(s)} J_{\Lambda} e_{j(t)}$ 
under the identification \[
\Hom_{\Lambda}(P_{s}, P'_{t}) \cong \Hom_{\Lambda}(\Lambda e_{i(s)}, \Lambda e_{j(t)}) \cong e_{i(s)}\Lambda e_{j(t)}.
\]

By   Lemma \ref{202309151536} and Lemma \ref{202309151607}, we see that $F$ and $RF$ preserve radical morphisms between projective modules. 
Thus we obtain the following lemma.

\begin{lemma}\label{202309281928}
Let $P, P'$ be projective $\Lambda$-modules.   
Then for a radical morphism $f : P \to P'$, 
the morphism $F(f): F(P) \to F(P')$ (resp. $RF(f): RF(P) \to RF(P')$) is a radical morphism 
between projective $A$ (resp. $\Lambda$)-modules. 
\end{lemma}

Let $M$ be a $\Lambda$-module 
and $P$ a projective $\Lambda$-module.  
Recall from \cite[Proposition I.4.1]{ARS} that a surjective homomorphism $p: P \to M$  is projective cover  if and only if  it is right minimal. 

The following lemma  shows that the functors $F, RF$ preserve projective covers.

\begin{lemma}\label{202309281244}
Let $ p: P \to M$ be a projective cover of a $\Lambda$-module $M$. 
Then 
\begin{enumerate}[(1)]
\item $F(p): F(P) \to F(M)$ is a projective cover. 

\item $RF(p): RF(P) \to RF(M)$ is a projective cover. 

\end{enumerate}
\end{lemma}

\begin{proof}
(1) Let $K=\Ker p$ be the kernel of $p$ and $i: K \to P$ the canonical inclusion. 
By Lemma \ref{202309281150}, $i \in \rad(K,P)$. 
By the same lemma, 
we only have to show that $F(i) \in \rad(F(K),F(P))$, 
since $F$ is exact and preserves projective modules.  

Let $q:P' \to K$ be a projective cover. Then $i\circ q \in \rad(P',P)$.
By   Lemma \ref{202309281928}, we have $F(i)\circ F(q) = F(i\circ q) \in\rad(F(P'),F(P))$. 
Since $F$ is exact, $F(q): F(P') \to F(K)$ is a surjective homomorphism.
Thus we can apply Lemma \ref{202309281219} and conclude that $F(i) $ belongs to $\rad(F(K),F(P))$.

(2) Since $RF$ is exact and  preserves projective modules and radical morphisms between projective modules, 
the same proof with (1) works for $RF$. 
\end{proof}

Let $M$ be a $\Lambda$-module. 
Recall that a projective presentation $P_{1} \xrightarrow{f } P_{0} \xrightarrow{ p}  M \to 0$ is called \emph{minimal}  
if the homomorphisms $p: P_{0} \to M$ and $f: P_{1} \to \Ker p$ are projective covers.

As a consequence of  Lemma \ref{202309281244}, we obtain the following corollary 
that shows that the functors $F, RF$ preserve minimal projective presentations.

\begin{corollary}\label{202309180059}
Let $M$ be a $\Lambda$-module and  $P_{1} \xrightarrow{f } P_{0} \to M \to 0$ a minimal projective presentation. 
Then, the following assertions hold.

\begin{enumerate}[(1)]

\item 
The projective presentation $F(P_{1}) \xrightarrow{ F(f) } F(P_{0}) \to  F(M) \to 0$ of the  $A$-module $F(M)$ is minimal.

\item 
The projective presentation $RF(P_{1}) \xrightarrow{ RF(f) } RF(P_{0}) \to RF(M) \to 0$ of the $\Lambda$-module $RF(M)$ is minimal. 

\end{enumerate}
\end{corollary}

\subsubsection{Compatibility of the AR-translation $\tau$ with $F$} 

It was proved that the Auslander-Reiten translations $\tau$ (of $\Lambda$ and of $A$) commutes with  the induction functor $F$ 
in the case that $|G|$ is invertible in $\kk$ by Reiten-Riedtmann \cite[Proof of Lemma 4.2]{RR}.
(Breaz-Marcus-Modoi \cite[Proposition 3.5]{BMM} proved the same commutativity for group graded algebras satisfying some conditions.)
In the next proposition, we show that the commutativity holds true in general.

\begin{proposition}[{\cite[Lemma 7.3]{CW1}}]\label{202309150753}
Let $M$ be a $\Lambda$-module. 
Then, $\tau F(M) = F\tau(M)$.
\end{proposition}

This proposition has been demonstrated in \cite[Lemma 7.3]{CW1} using the same method as our proof. However, for the reader's convenience, we include a proof here.

We need the following compatibility of $F=A\otimes_{\Lambda}-$ with the $\kk$-duality $\tuD(-)=\Hom_{\kk}(-,\kk)$. 

\begin{lemma}\label{202309150735}
For a right $\Lambda$-module $M'$, we have the following   isomorphism of  left $A$-modules 
which is natural in $M'$
\[
A \otimes_{\Lambda} \tuD(M')  \cong \tuD(M' \otimes_{\Lambda}A)
\]
\end{lemma}

\begin{proof}
The desired isomorphism is obtained from the following string of isomorphisms: 
\[
\begin{split}
A \otimes_{\Lambda}  \tuD(M')  
& \cong \Hom_{\Lambda}(A, \Lambda) \otimes_{\Lambda}\tuD(M')  \cong \Hom_{\Lambda}(A, \tuD(M')) \\ 
& \cong \Hom_{\kk}(M' \otimes_{\Lambda} A, \kk)
 = \tuD(M' \otimes_{\Lambda} A).
\end{split}
\]
\end{proof}

For simplicity we use the same symbol $(-)^{*}$ to denote the dualities 
$(-)^{*}:=\Hom_{\Lambda}(-,\Lambda)$ and $(-)^{*} := \Hom_{A}(-,A)$

\begin{proof}[Proof of Proposition \ref{202309150753}]
Let $P_{1} \xrightarrow{f} P_{0} \to M \to 0$ be a minimal projective presentation of $M$. 
Then we have $\tau(M) \cong \Ker [D(f^{*}) : \tuD(P_{1}^{*}) \to \tuD(P_{0}^{*})] $ 
and  by Lemma \ref{202309150735}
\[
\begin{split}
A \otimes_{\Lambda} \tau(M) 
& \cong \Ker [A \otimes D(f^{*}): A\otimes_{\Lambda} \tuD(P_{1}^{*}) \to A \otimes_{\Lambda} \tuD(P_{0}^{*})] \\
& \cong \Ker [D(f^{*} \otimes A): \tuD(P_{1}^{*}\otimes_{\Lambda}  A) \to  \tuD(P_{0}^{*} \otimes_{\Lambda}A)] \\
& \cong \Ker [D( (A \otimes f)^{*} ): \tuD((A \otimes_{\Lambda} P_{1})^{*}) \to  \tuD((A \otimes_{\Lambda} P_{0})^{*})] \\
\end{split}
\]
where for the last isomorphism we use the isomorphism 
$\Hom_{\Lambda}(P, \Lambda) \otimes_{\Lambda} A \cong \Hom_{A}(A\otimes_{\Lambda} P, A)$
for a   $\Lambda$-module $P$. 

By Lemma \ref{202309180059}(1), the projective presentation 
 $A\otimes_{\Lambda} P_{1} \xrightarrow{ A\otimes f } A \otimes_{\Lambda} P_{0} \to A\otimes_{\Lambda} M \to 0$ of $A \otimes_{\Lambda} M$ is  minimal. 
It follows that $\tau(A \otimes_{\Lambda} M) = \Ker D((A\otimes f)^{*}) \cong A \otimes_{\Lambda} \tau(M)$.
\end{proof}

\subsubsection{Compatibility of the AR-translation $\tau$ with $R$}

It is well-known that 
if $\chara \kk \nmid |G|$, 
the Auslander-Reiten translations $\tau$ (of $\Lambda$ and of $A$) 
 commutes with  the restriction functor $R$. 
(Dade \cite[Corollary 5.7]{Dade: extending} (see also \cite[Proposition 3.7]{BMM}) proved the commutativity for group graded algebras satisfying certain condition.) 

We remark that the commutativity does not hold true in general. 

\begin{example}
For example, let $\kk = \FF_{2}, G= \ZZ/2\ZZ$ and $\Lambda =\kk$ with the trivial $G$-action. 
Then $A = \kk G \cong \kk[x]/(x^{2})$ and for a trivial $G$-module $N := \kk$ 
we have $R\tau (N) = R(N)$ and $\tau R(N) = 0$. Thus $\tau R(N) \neq R\tau(N)$.
\end{example}

In the next proposition, we show that if we put a mild condition (irrelevant to $|G|$ nor $\chara \kk$) on an $A$-module $N$, 
then we have $\tau R(N) \cong R\tau(N)$.

\begin{proposition}\label{2023091512311}
Let $N$ be an $A$-module and $Q_{1} \xrightarrow{g} Q_{0} \to N \to 0$ be a minimal projective presentation.  
Assume that $N \in \add FR(N)$. 
Then the following assertions hold.

\begin{enumerate}[(1)]

\item The morphism $R(g): R(Q_{1}) \to R(Q_{0})$ gives a minimal projective presentation of $R(N)$.

\item The morphism $FR(g): FR(Q_{1}) \to FR(Q_{0})$ gives a minimal projective presentation of $FR(N)$. 

\item We have 
$\tau R(N)\cong R\tau(N)$.

\end{enumerate} 
\end{proposition}

We note that 
by Lemma \ref{202404031821}, 
in the case $\chara \kk \nmid |G|$, we have $N \in \add FR(N)$ for any $A$-module $N$.  
Hence  we can  recover  the  commutativity $\tau R(N) \cong  R(\tau N)$ in the classical case $\chara \kk \nmid |G|$ from the above proposition.  
We also note that thanks to Lemma \ref{202309150913}, a \textup{(}$\mod G$\textup{)}-stable rigid $A$-module $N$ satisfies the condition $N \in \add FR(N)$.

To prove the proposition, we need the following compatibility of $R$ with dualities $(-)^{*}=\Hom_{\Lambda}(-,\Lambda)$ and $(-)^{*} = \Hom_{A}(-,A)$

\begin{lemma}\label{202309151225}
For an $A$-module $N$, there is an isomorphism $R(N^{*}) \cong R(N)^{*}$ which is natural in $N$.
\end{lemma}

\begin{proof}
Since $F$ is a right adjoint of $R$, we have the following isomorphism of $\kk$-vector spaces 
\[
\Hom_{A}(N,A)  \cong \Hom_{A}(N, F(\Lambda)) \cong \Hom_{\Lambda}(R(N), \Lambda).  
\]
It is straightforward to check that this map is compatible with the action of $\Lambda$. 
\end{proof}

\begin{proof}[Proof of Proposition \ref{2023091512311}]
(1)
We set $M := R(N)$. Let $P_{1} \xrightarrow{ f } P_{0} \to M \to 0$ be a minimal projective presentation of $M$. 
It follows from the assumption  $N \in \add F(M)$ that 
the complex $[Q_{1} \to Q_{0}]$ is a direct summand of a direct sum of  the complex $[F(P_{1})  \to F(P_{0})]$. 
Hence the complex $[R(Q_{1}) \to R(Q_{0})]$ is a direct summand of a direct sum of the complex 
$[RF(P_{1}) \to RF(P_{0})]$. 
On the other hand, by Corollary \ref{202309180059}, 
the projective presentation $RF(P_{1}) \xrightarrow{ RF(f) } RF(P_{0}) \to RF(M) \to 0$ is minimal. 
It is easy to see that if a projective presentation is a direct summand of a minimal projective presentation, then it is a minimal projective presentation. 
Thus we conclude that $R(g)$ gives a minimal projective presentation of $R(N)$. 

(2) follows from Corollary \ref{202309180059}.

(3)  We have the following isomorphisms by Lemma \ref{202309151225}
\[
\begin{split}
R \tau(N) 
& \cong \Ker [R\tuD(g^{*}): R\tuD(Q_{1}^{*}) \to R\tuD(Q_{0}^{*})] \\
& \cong \Ker [\tuD(R(g^{*})): \tuD(R(Q_{1}^{*})) \to \tuD(R(Q_{0}^{*}))] \\
& \cong \Ker [\tuD((Rg)^{*}): \tuD( R(Q_{1})^{*}) \to \tuD( R(Q_{0})^{*}) ]. 
\end{split}
\]
By (1),   $R(Q_{1}) \xrightarrow{ R(g) } R(Q_{0}) \to R(N) \to 0$ is a minimal projective presentation of $R(N)$. 
It follows that  $R\tau(N) \cong \Ker \tuD((Rg)^{*}) =\tau R(N)$. 
\end{proof}
%
%
%

%
%
%
%
%
%
%
%
%
%
%
%
%
%
%
%
%
%

\section{$\tau$-tilting theory of skew group algebra extensions}\label{202404102224}

In this section we discuss $\tau$-tilting theory of a skew group algebra extension $\Lambda \subset \Lambda *G$. 
For the definition of support $\tau$-tilting modules and other basic notions and results, 
we refer \cite{AIR}.

\subsection{The bijection}

The next theorem is the first step to prove Theorem \ref{202309301601}, 
which is the $\tau$-tilting part of Main Theorem \ref{202404091842}. 

\begin{theorem}\label{202309151712}

\begin{enumerate}[(1)]

\item If $M$ is a $G$-stable support $\tau$-tilting $\Lambda$-module, then $F(M)$ is a support $\tau$-tilting $A$-module. 

\item If $N$ is a \textup{(}$\mod G$\textup{)}-stable support $\tau$-tilting $A$-module, then $R(N)$ is a support $\tau$-tilting $\Lambda$-module. 
\end{enumerate}
\end{theorem}

\begin{proof}[A sketch of the first proof of Theorem \ref{202309151712}.]
(1) is a generalization of  \cite[Theorem 3.2]{Koshio}  in our setting 
as well as a generalization of \cite[Theorem 4.2]{ZH}. 
With the preparations given before mainly the commutativity $\tau F(M) \cong F(\tau M)$ (Proposition \ref{202309150753}), 
the same proof  with \cite{ZH} works. 

(2) is a generalization of 
 \cite[Theorem 3.4]{KK} in our setting.  
With the preparations given before, 
mainly the commutativity $\tau R(N) \cong R(\tau N)$ for \textup{(}$\mod G$\textup{)}-stable $N$ (Proposition \ref{2023091512311}), 
the same proof  with \cite{KK} works.
\end{proof}

\begin{remark}
In Section  \ref{202403121920} we give an alternative proof of Theorem \ref{202309151712} by using $2$-term silting complexes. 
In the proof we do not use the commutativity $\tau F(M)\cong F(\tau M), \ \tau R(N)\cong R(\tau N)$ that played crucial roles in the first proof. 
\end{remark}

%
%
%
%
%
%
%
%
%
%
%

Let $\sttilt \Lambda, \ \sttilt A$ be the posets of (isomorphism classes of)  basic support $\tau$-tilting modules over $\Lambda$ and $A$ respectively. 
We denote by $(\sttilt \Lambda)^{G}$ the subposet of  $G$-stable basic support $\tau$-tilting $\Lambda$-modules, 
and by $(\sttilt A)^{\mod G}$ the poset of \textup{(}$\mod G$\textup{)}-stable basic support $\tau$-tilting $A$-modules.

We remark that since  we are not assuming that our algebras $\Lambda, A$ are  basic,  strictly speaking 
 the $\tau$-tilting module $\Lambda$ (resp. $A$) does not necessarily belong to $\sttilt \Lambda$ (resp. $\sttilt A$). 
 However, to simplify the notation, we often denote $\Lambda$ and $A$ to denote the basic $\tau$-tilting modules $\Lambda^{\basic}$ 
 and $A^{\basic}$. 

\begin{theorem}[{cf. \cite[Corollary 3.9]{KK}}]\label{202309301601}
The adjoint pair $(F=\Ind_{\Lambda}^{A},R=\Res_{\Lambda}^{A})$ induces an isomorphism of posets 
\[
F= \Ind_{\Lambda}^{A}: (\sttilt \Lambda)^{G} \longleftrightarrow (\sttilt A)^{\mod G} :\Res_{\Lambda}^{A}=R.
\]
\end{theorem}

\begin{proof}
Thanks to Theorem \ref{202309151712}, the maps 
$F: (\sttilt \Lambda)^{G} \to (\sttilt A)^{\mod G}, \ $  $ M \mapsto F(M)^{\basic}$ and 
$R: (\sttilt A)^{\mod G} \to (\sttilt \Lambda)^{G}, \ N \mapsto R(N)^{\basic}$ are well-defined. 
Moreover, it follows from Corollary  \ref{202310021718} and Proposition \ref{202309150913} 
 that these maps are the  inverse maps to each other. 

Recall that for  two support $\tau$-tilting modules $M_{1}, M_{2}$, the inequality 
$M_{1} \geq M_{2}$ is defined to be existence of a surjection $M'_{1} \twoheadrightarrow M_{2}$ from some $M'_{1} \in \add M_{1}$. 
Since the functors $F, R$ preserve surjections, it is straightforward to check that the above maps are poset morphisms.  
\end{proof}

Recall that a finite abelian group $G$ is said to split over $\kk$ if all its simple modules are $1$-dimensional. 
It follows that  the case that a finite group $G$ is abelian and splits over $\kk$,  
\textup{(}$\mod G$\textup{)}-stability of $A$-modules coincides with $\XX$-stability of $A$-modules. 
Thus we obtain the following corollary in which we denote the poset of $\XX$-stable support $\tau$-tilting $A$-modules by $(\sttilt A)^{\XX}$. 

\begin{corollary}\label{202403290935}
Assume that $G$ is a finite abelian group which splits over $\kk$. Then the adjoint pair $(\Ind_{\Lambda}^{A},\Res_{\Lambda}^{A})$ induces an isomorphism of posets 
\[
\Ind_{\Lambda}^{A}: (\sttilt \Lambda)^{G} \longleftrightarrow (\sttilt A)^{\XX} :\Res_{\Lambda}^{A}.
\]
\end{corollary}

\begin{remark}\label{202403290936}
We note that if $\kk$ is algebraically closed, then every finite abelian group $G$ splits over $\kk$. 
In the case that $\kk$ is algebraically closed, 
the statement of Corollary \ref{202403290935}  for an abelian group $G$ such that $\chara \kk \nmid |G|$ 
was claimed in \cite[Corollary 4.7]{ZH}, which was obtained as a consequence of 
\cite[Theorem 4.6]{ZH} that claimed the same statement holds true even in the case that $G$ is solvable
such that $\chara \kk \nmid |G|$.
However, as is shown in Example \ref{202309271835}, there exists an example of an algebra $\Lambda$ with an action $G\curvearrowright \Lambda$ 
of a solvable group $G$ such that $\chara \kk \nmid |G|$ 
 that satisfies  $(\sttilt A)^{\mod G} \subsetneq (\sttilt A)^{\XX}$ where we set $A:= \Lambda *G$. 
 Consequently, the map $\Ind_{\Lambda}^{A}: (\sttilt \Lambda)^{G} \to (\sttilt A)^{\XX}$ is not a bijection.  
\end{remark}

\subsubsection{An example}

\begin{example}\label{202309271835}

Assume that $\kk$ is algebraically closed. 

Let $G=\frkS_{3}$ be the symmetric group of degree $3$ and $\Lambda := \kk[x]/(x^{2})$. 
We equip an action of $G$ on $\Lambda$ to be $\sigma\cdot x:= \sgn(\sigma) x$ for $\sigma \in S_{3}$. 
We set $A:= \Lambda * G$. 

Since the algebra $\Lambda$ is local, we have $\sttilt \Lambda = \{ 0 < \Lambda\}$. 
It is clear that both of $\Lambda, 0$ are $G$-stable. Hence we have $(\sttilt \Lambda)^{G} =\{0 < \Lambda\}$. 
It follows from Theorem \ref{202309301601}  that $(\sttilt A)^{\mod G} = \{ 0 < A\}$. 

From now, we directly compute the Hasse quiver of $\sttilt A$, 
determine the actions $S\otimes_{\kk}^{G}-$ of simple $G$-module on $\sttilt A$ and finally check that $(\sttilt A)^{\mod G}=\{ 0< A\}$ actually holds. 
We see that the $\tau$-tilting theory of $A$ is heavily depends on the characteristic of the base fields.

\noindent 
(I) The case $\chara \kk \neq 2,3$.

To describe the skew group algebra $A$, we set  $Q$ to  be the quiver below and  $A'_{12}:= \kk Q/(x'x'', x''x')$.
\[
\begin{xymatrix}{
1 \ar@<2pt>[r]^{x'} & 2\ar@<2pt>[l]^{x''}
}\end{xymatrix}
\]
Then using the general method given in  \cite[Section 3]{BSW} (or direct computation), 
we can check that $A$ is isomorphic to $A'_{12} \times M_{2}(\Lambda)$ and hence Morita equivalent to $A':=A'_{12} \times \Lambda$. 
We set  $e_{3}:=(0,1_{\Lambda}) \in A'$. 

The Hasse diagram of the poset $\sttilt A=\sttilt A'$ is given below: 
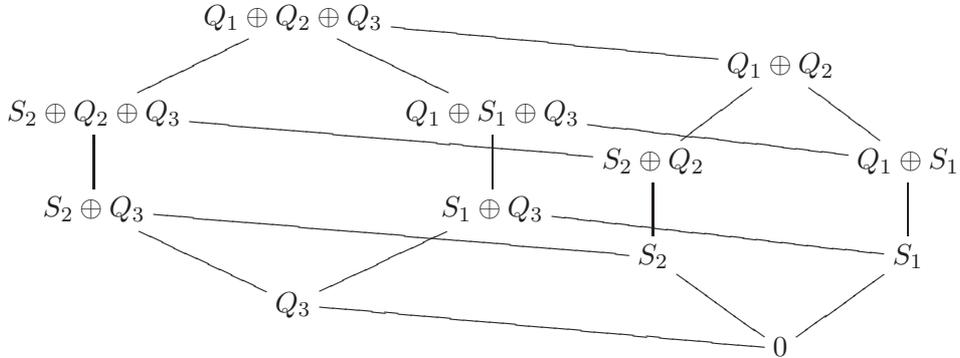
\[\begin{xymatrix}@R=2pt@C=2pt{
& Q_{1}\oplus Q_{2} \oplus Q_{3} \ar@{-}[ddl] \ar@{-}[ddr] \ar@{-}[drrr] &&&&\\
&&&& Q_{1} \oplus Q_{2} \ar@{-}[ddl] \ar@{-}[ddr] & \\
 S_{2} \oplus Q_{2} \oplus Q_{3} \ar@{-}[dd] \ar@{-}[drrr] && Q_{1} \oplus S_{1} \oplus Q_{3}  \ar@{-}[dd] \ar@{-}[drrr]&&& \\
 &&& S_{2} \oplus Q_{2}\ar@{-}[dd] && Q_{1} \oplus S_{1}\ar@{-}[dd] \\
 S_{2}  \oplus Q_{3} \ar@{-}[drrr] && S_{1} \oplus Q_{3} \ar@{-}[drrr]&&& \\
 &&& S_{2}  &&   S_{1} \\
 &Q_{3} \ar@{-}[uul] \ar@{-}[uur]  \ar@{-}[drrr]&&&& \\
 &&&& 0  \ar@{-}[uul] \ar@{-}[uur] 
}\end{xymatrix}
\]
where $Q_{1}, Q_{2}, Q_{3}$ denote the indecomposable projective $A'$-modules corresponding to the vertices  $1,2,3$ 
and $S_{1}, S_{2}$ denote the simple $A'$-modules corresponding to the vertices  $1,2$. 

We study \textup{(}$\mod G$\textup{)}-stability and $\XX$-stability. 
Let $\kk_{G}$ be the trivial representation, $\kk_{\sgn}$ the sign representation and $V$ the unique irreducible $2$-dimensional representation of $G$. 
Using Example \ref{202309271819} we can check that 
\begin{equation}\label{202309291219}
\begin{split}
\kk_{\sgn}\otimes_{\kk}^{G} Q_{1} \cong Q_{2}, \ \kk_{\sgn} \otimes_{\kk}^{G} Q_{2} \cong Q_{1}, \ \kk_{\sgn}\otimes_{\kk}^{G} Q_{3} \cong Q_{3}, \\
\kk_{\sgn}\otimes_{\kk}^{G} S_{1} \cong S_{2}, \ \kk_{\sgn} \otimes_{\kk}^{G} S_{2} \cong S_{1}, \ \kk_{\sgn}\otimes_{\kk}^{G} S_{3} \cong S_{3}, \\
V\otimes_{\kk}^{G} Q_{1} \cong Q_{3}, \ V \otimes_{\kk}^{G} Q_{2} \cong Q_{3}, \ V\otimes_{\kk}^{G} Q_{3} \cong Q_{1} \oplus Q_{2} \oplus Q_{3}. 
\end{split}
\end{equation}

Thus, \textup{(}$\mod G$\textup{)}-stable support $\tau$-tilting modules are precisely $Q_{1} \oplus Q_{2} \oplus Q_{3}$  and $0$. 
These are $\XX$-stable. But there are two more $\XX$-stable support $\tau$-tilting modules $Q_{3}$ and $Q_{1} \oplus Q_{2}$. 

\begin{remark}
The above isomorphisms \eqref{202309291219} are that in the category $A'\mod$. 
Thus, in particular, we regard $\kk_{\sgn} \otimes_{\kk}^{G}-, \ V\otimes_{\kk}^{G}-$ are endofunctors of $A'\mod$ 
via the equivalence $(A'\times \Lambda)\mod \cong \mod A$. 
The same remark is applied to the isomorphisms \eqref{2023092912192}.
\end{remark}

\noindent 
(II) The case $\chara \kk = 2$. 
In this case, the action $G \curvearrowright \Lambda$ is trivial. 
It follows that  the skew group algebra $A$ is the ordinary group algebra $A= \Lambda \otimes \kk G$.
It is well-known that $\kk G \cong \kk[y]/(y^{2}) \times M_{2}(\kk)$. 
Hence we see that $A \cong \Lambda[y]/(y^{2}) \times M_{2}(\Lambda)$ and that 
$A$ is Morita equivalent to $A'':=\Lambda[y]/(y^{2}) \times \Lambda= \kk[x,y]/(x^{2},y^{2}) \times \kk[x]/(x^{2})$.

The Hasse diagram of the poset $\sttilt A=\sttilt A''$ is given below: 
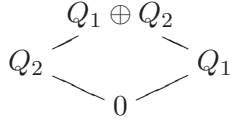
\[\begin{xymatrix}@R=2pt@C=2pt{
& Q_{1} \oplus Q_{2} \ar@{-}[dl] \ar@{-}[dr] & \\
 Q_{2}\ar@{-}[dr] && Q_{1} \ar@{-}[dl] \\
 & 0  
}\end{xymatrix}
\]
where $Q_{1}$ and  $Q_{2}$ denote the indecomposable projective $A''$-modules corresponding to the primitive idempotent elements
 $e_{1} := (1,0), e_{2}:= (0,1)$ of $A''$. 

Since we are assuming $\chara \kk =2$, the symmetric group $G=\frkS_{3}$ has two simple modules, the trivial module $\kk_{G}$ and 
the unique  $2$-dimensional irreducible representation $V$.  
It follows that every $A$-module is $\XX$-stable. 
We can check that 
\begin{equation}\label{2023092912192}
\begin{split}
V \otimes_{\kk}^{G} Q_{1} \cong Q_{2}^{\oplus 2}, \ V \otimes_{\kk}^{G} Q_{2} \cong Q_{1}\oplus Q_{2}. 
\end{split}
\end{equation}
Thus, \textup{(}$\mod G$\textup{)}-stable support $\tau$-tilting modules are $Q_{1} \oplus Q_{2}$  and $0$. 

\noindent 
(III) The case $\chara \kk = 3$. 

Let $\eta, \zeta \in G$ be elements of order $3$ and $2$ respectively. 
Then we have $G \cong \agl{\eta} \rtimes \agl{\zeta}$. 
Since $\eta$ trivially acts on $\Lambda$, we have  
$A = \Lambda * G \cong (\Lambda * \agl{\eta}) * \agl{\zeta}\cong (\Lambda\otimes\kk[y]/(y^{3})) *\agl{\zeta}\cong (\kk[x,y]* \agl{\zeta} )/(x^{2}, y^{3})$ where in the right most term, we take the  action $\agl{\zeta} \curvearrowright \kk[x,y]$ given by $\zeta\cdot x: =-x, \ \zeta\cdot y:= -y$. 
It is well-known that the algebra $\kk[x,y]* \agl{\zeta}$ is isomorphic to the preprojective algebra $\Pi(\tilde{\AA}_{1})$ of the extended Dynkin quiver $\tilde{\AA}_{1}$. 
It  follows that the skew group algebra $A$ is isomorphic to the following path algebra $\kk Q''/I$ with the relations 
\[
Q'': 
\begin{xymatrix}@R=40pt{
1 \ar@/^10pt/@<10pt>[rr]^{x}\ar@<5pt>@/^8pt/[rr]_{y} & &
2 \ar@/^10pt/@<10pt>[ll]^{x'}\ar@<5pt>@/^8pt/[ll]_{y'}, 
}\end{xymatrix}
\] 
and  the relations are $xx'=0, x'x= 0, yy'y=0, y'yy'=0, x'y=y'x, xy'=yx'$.

The Hasse quiver of the poset $\sttilt A$ has two connected components both of which are infinite: 
\begin{equation}\label{202403291436}
\begin{xymatrix}@R=2pt@C=2pt{
& Q_{1} \oplus Q_{2} \ar@{-}[ddl] \ar@{-}[ddr] & \\
& \ar@{--}[ddddddd]&& \\
N_{(2)} \oplus Q_{2}\ar@{-}[dd] && Q_{1} \oplus N'_{(2)}\ar@{-}[dd] \\
&& \\
N_{(2)} \oplus N_{(3)} \ar@{-}[dd] &&  N'_{(3)}\oplus N'_{(2)} \ar@{-}[dd]  \\
&& \\
N_{(4)} \oplus N_{(3)} \ar@{-}[dd] &&  N'_{(3)}\oplus N'_{(4)} \ar@{-}[dd]  \\
&& \\
\vdots && \vdots \\
}\end{xymatrix} 
\begin{xymatrix}@R=2pt@C=2pt{
\vdots \ar@{-}[dd] &\ar@{--}[ddddddddd]& \vdots \ar@{-}[dd]\\
&&&&\\ 
L_{(3)} \oplus L_{(4)}   \ar@{-}[dd] && L'_{(3)} \oplus L'_{(4)} \ar@{-}[dd] \\
&& \\
L_{(3)} \oplus L_{(2)}   \ar@{-}[dd] && L'_{(3)} \oplus L'_{(2)} \ar@{-}[dd] \\
&& \\
S_{1} \oplus L_{(2)}  \ar@{-}[dd] && S_{2} \oplus L'_{(2)} \ar@{-}[dd] \\
&& \\  
S_{1} \ar@{-}[ddr] && S_{2} \ar@{-}[ddl]\\
&& \\
&0&
}\end{xymatrix}
\end{equation}
where 
\[
\begin{split}&
N_{(2)}:=\begin{smallmatrix}
&2 & & 2& &\\
1& &1 &&1&\\
&&&2&&2
\end{smallmatrix}, 
N_{(3)}:=\begin{smallmatrix}
&2&&2 & & 2& &\\
1&&1& &1 &&1&\\
&&&&&2&&2
\end{smallmatrix}, 
N_{(4)}:=\begin{smallmatrix}
&2 &&2&&2 & & 2& &\\
1 &&1&&1& &1 &&1&\\
&&&&&&&2&&2
\end{smallmatrix}, \dots \\
&
L_{(2)} := 
\begin{smallmatrix} 
2 & & 2 \\ 
& 1 
\end{smallmatrix}, 
L_{(3)} := 
\begin{smallmatrix} 
2 & & 2 &&2 \\ 
& 1  && 1
\end{smallmatrix}, 
L_{(4)} := 
\begin{smallmatrix} 
2&&2 & & 2 &&2 \\ 
& 1&& 1  && 1
\end{smallmatrix}, \dots 
\end{split}
\]
and the composition series of $N'_{(i)}, L'_{(i)}$ for $i=2,3,\dots$  are obtained from that of $N_{(i)}, L_{(i)}$ by replacing the composition factors $1,2$ 
with $2,1$. 

We can show by the same method of the proof of \cite[Theorem 6.17]{DIJ}  
that the Hasse quiver of $\sttilt A$ has precisely two components given in \eqref{202403291436}.  
(We have the following alternative proof.  The element $\eta:= yy' +y'y$ of $A$ is central and belongs to the Jacobson radical.
It follows from \cite[Theorem 4.1]{EJR} that the canonical map $\sttilt A \to \sttilt A/(\eta)$ is bijective.  
On the other hand, the algebra $A/(\eta)$ is precisely the algebra dealt in \cite[Theorem 6.17]{DIJ}.  
Hence  we can obtain the Hasse quiver of  the poset $\sttilt A/(\eta) \cong \2silt A/(\eta)$.) 

Since we are assuming $\chara \kk =3$, the symmetric group $G=\frkS_{3}$ has two simple modules, the trivial module $\kk_{G}$ and the sign representation $\kk_{\sgn}$. It follows in particular that \textup{(}$\mod G$\textup{)}-stability coincides with $\XX$-stability. 
We can check that 
\[
\begin{split}
\kk_{\sgn}\otimes_{\kk}^{G} Q_{1} \cong Q_{2}, \ \kk_{\sgn} \otimes_{\kk}^{G} Q_{2} \cong Q_{1},
\kk_{\sgn}\otimes_{\kk}^{G} S_{1} \cong S_{2}, \ \kk_{\sgn} \otimes_{\kk}^{G} S_{2} \cong S_{1}.
\end{split}
\]
It follows that the functor $\kk_{\sgn}\otimes^{G}_{\kk}-$ induces the reflection with respect to the vertical dotted lines of the Hasse quiver of the poset $\sttilt A$.

Thus, we check that \textup{(}$\mod G$\textup{)}-stable support $\tau$-tilting modules are $Q_{1} \oplus Q_{2}$  and $0$. 
\end{example}

\subsection{$\tau$-tilting finiteness}\label{202404091141}

Recall that an algebra $\Lambda$ is said to be \emph{$\tau$-tilting finite} 
if $|\sttilt \Lambda| < \infty$. 
From Example \ref{202309271835}(III) we observe that $\tau$-tilting finiteness is not necessarily preserved by a skew group algebra extension. 
The aim of this section is to prove Theorem \ref{202309301357} that tells that finiteness of support $\tau$-tilting modules of $\Lambda$ is inhered to $A = \Lambda * G$ under certain condition that can be easily checked in many situations.

\subsubsection{Preparations}

We need preparations.
In many situation of applications, an algebra $\Lambda=\kk Q/(r_{1}, \dots, r_{n})$ is often given by a quiver $Q$ with the relations $r_{1}, \dots, r_{n}$ 
and an action $G\curvearrowright \Lambda$  is often  induced from an action $G\curvearrowright Q$ on the quiver $Q$ 
that preserves the ideal $(r_{1}, \dots, r_{n})$. 
The following definition is an abstraction of this situation.

\begin{definition}\label{202309301259}
Let $\Lambda$ be an algebra with a $G$-action. 

We say that the action $G \curvearrowright \Lambda$ \emph{preserves idempotent elements} 
if the following conditions (1)(2) hold. 
We say that  the action $G \curvearrowright \Lambda$ \emph{preserves idempotent elements and has local stabilizers}  
if the following conditions (1)(2)(3)  hold. 
\begin{enumerate}[(1)]

\item 
A finite dimensional  algebra $\Lambda$ is  basic and elementary 
with a  complete set $\sfE \subset \Lambda$ of primitive orthogonal idempotent elements. 
Namely, for $e, e' \in \sfE, \ e\neq e'$, we have $e(\Lambda/J_{\Lambda} )e' = 0$ and $e (\Lambda/J_{\Lambda}) e = \kk$. 

\item 
The action of $G$ preserves $\sfE$. 
Namely,  
for any  $e \in \sfE$ and $g \in G$, we have $g(e) \in \sfE$

\item For any $e \in \sfE$, the group algebra $\kk \stab(e)$ of  the stabilizer group $\stab(e)$ is local.  
\end{enumerate}

\end{definition}

As is mentioned above, 
the following is an elementary and basic example of 
an action $G \curvearrowright \Lambda$ that preserves idempotent elements. 

\begin{example}
Let $Q$ be a finite quiver and $I$ a two-sided admissible ideal of $\kk Q$. 
If a finite group $G$ acts on $Q$ in such a way that the extended action on $\kk Q$ preserves $I$. 
Then the induced action of $G$ on $\Lambda := \kk Q/I$ preserves the canonical  idempotent elements 
$\sfE:= \{ e_{i} \mid i \in Q_{0}\}$ that correspond to the vertices  of $Q$. 
\end{example}

The following lemma provides basic properties of 
indecomposable projective modules over $A =\Lambda*G$ 
in the case that  an action $G \curvearrowright \Lambda$  preserves idempotent elements.

\begin{lemma}\label{202309301224}
Let $\Lambda$ be an algebra with a $G$-action that preserves idempotent elements and $\sfE$ the fixed set in  Definition \ref{202309301259}.
Then the following assertions hold. 

\begin{enumerate}[(1)]
\item 
Let $e, e' \in \sfE$. 
Then the idempotent elements $e, e'$ are equivalent as idempotent elements of $A$ 
if and only if these elements are in the same $G$-orbit.

\item 
 An element $e \in \sfE$ (regarded as an element of $A$ via the canonical inclusion $\Lambda \subset A$) is primitive in $A$
  if and only if  the group algebra $\kk \stab(e)$ is local where  $\stab(e)$ is the stabilizer subgroup of $e$.  
 
 \item The projective module $A e$ is \textup{(}$\mod G$\textup{)}-stable for each $e\in \sfE$. 
 
 \end{enumerate}
\end{lemma}

\begin{proof}
(1)
``If'' part. 
We assume that there exists $g \in G$ such that $e' = g(e)$. 
Then the elements $a := e*g^{-1}, \ b:= g(e) * g$ satisfy the equations 
\begin{equation}\label{202309301535}
ea= ae' =a, \ e'b=be =b, \ ab= e, \ ba = e'.
\end{equation}

``Only if'' part. 
We prove only if part by contradiction. 
Assume that $e$ and $e'$ belong to different $G$-orbit. 
Let $a, b\in A$ that satisfy the equations \eqref{202309301535}. 
We write $a = \sum_{g\in G} a_{g} * g, \ b= \sum_{g\in G} b_{g} *g$ with $a_{g}, b_{g} \in \Lambda$. 
It follows from the first and second equations of \eqref{202309301535} that 
$a_{g} \in e \Lambda g(e')= eJ_{\Lambda} g(e'), \ b_{g} \in e'\Lambda g(e) =e'J_{\Lambda} g(e)$. 
Therefore, $e= ab, e' =ba $ belong to $J_{\Lambda} * G$. In particular, we come to the contradiction that  these elements are nilpotent.

(2) 
The two-sided ideal $\agl{J_{\Lambda}} :=AJ_{\Lambda} A $ of $A$ generated by $J_{\Lambda} \subset \Lambda \subset A$ 
is the linear span of the  elements of the forms $j * g \ (   j \in J_{\Lambda}, g \in G)$. 
It follows that $A/\agl{J_{\Lambda}} \cong \Lambda/J_{\Lambda} * G$. 
Observe that for $r\in \Lambda, g\in G$, we have $e(\overline{r} * g)e = (e\overline{r} g(e)) *g$ where $\overline{r} := r + J_{\Lambda}$. 
It follows from the  condition (1) of Definition \ref{202309301259} that $e(A/\agl{J_{\Lambda}}) e \cong \kk\stab(e)$. 

Since $\agl{J_{\Lambda}} \subset J_{A}$ by Lemma \ref{202309151536},
the algebra $e(A/J_{A}) e$ is the quotient algebra of $e(A/\agl{J_{\Lambda}})e$ by its Jacobson radical $e (J_{A}/\agl{J_{\Lambda}})e$.  
It follows that  
$\kk \stab(e) \cong e(A/\agl{J_{\Lambda}})e$ is local if and only if $e( A/J_{A}) e$ is a field 
if and only if $e$ is primitive as an element of $A$. 

(3) We fix a complete set $R$ of the representatives of $G/\stab(e)$ that contains the unit $e_{G}$ of $G$. 
Then $\bigoplus_{g \in R} \Lambda g(e)$ is a direct summand of the $A$-module $\Lambda$ 
and $\bigoplus_{g\in R}A g(e) =(\bigoplus_{g \in R} \Lambda g(e)) *G$.  
Hence $\bigoplus_{g \in R} Ag(e)$ is \textup{(}$\mod G$\textup{)}-stable. 
By (1), we have $\add A e= \add(\bigoplus_{g \in R} Ag(e))$. 
Therefore, we conclude that $Ae$ is \textup{(}$\mod G$\textup{)}-stable. 
\end{proof}

\subsubsection{Finiteness theorem}

\begin{theorem}\label{202309301357}
Let $G$ be a finite  group and  $\Lambda$  an algebra with $G$-action that preserves idempotent elements and has  local stabilizers. 
Moreover, we assume  
$|(\sttilt \Lambda)^{G}|< \infty$.  
 Then the following assertions hold.
 
 \begin{enumerate}[(1)]
 
\item  $A = \Lambda * G$ is  $\tau$-tilting finite.

\item All supprot $\tau$-tilting $A$-modules are \textup{(}$\mod G$\textup{)}-stable.

\item The adjoint pair $(\Ind_{\Lambda}^{A}, \Res_{\Lambda}^{A})$ induces a bijection 
\[
\Ind_{\Lambda}^{A}: ( \sttilt \Lambda)^{G} \longleftrightarrow \sttilt A: \Res_{\Lambda}^{A}.
\]
\end{enumerate}
\end{theorem}

To prove this theorem, we introduce the notion of ind-\textup{(}$\mod G$\textup{)}-stability.

\begin{definition}\label{202403311650}
A   support $\tau$-tilting $A$-module $N$ is said to be  
\emph{ind-\textup{(}$\mod G$\textup{)}-stable} if  each indecomposable direct summand is \textup{(}$\mod G$\textup{)}-stable. 
Namely, 
if  $N=\bigoplus_{i=1}^{n} N_{i}$ is  a decomposition into indecomposable modules $N_{i}$, 
then each $N_{i}$ is \textup{(}$\mod G$\textup{)}-stable. 

We note the  condition is equivalent to  $X \otimes_{\kk}^{G} N_{i} \cong N_{i}^{\oplus \dim X}$ for 
any $X \in \mod G$. 
\end{definition}

  A key step is the following lemma that states that ind-\textup{(}$\mod G$\textup{)}-stability of support $\tau$-tilting modules 
  is preserved by irreducible left mutations.

\begin{lemma}\label{202309301439}
If $N=\bigoplus_{i=1}^{n} N_{i}$ is a basic  ind-\textup{(}$\mod G$\textup{)}-stable support $\tau$-tilting $A$-module with decompostion into  indecmoposable direct summand  $N_{i}$. 
Then  for any $i=1,\dots, n$, the irreducible left mutation $\mu_{i}^{-}(N)$  is also ind-\textup{(}$\mod G$\textup{)}-stable. 
\end{lemma}

\begin{proof}
The case $n=1$ is clear. Thus we may assume that $n\geq 2$. 
We may also assume that $i=1$. We set $N':= \bigoplus_{i=2}^{n} N_{i}$. 
Let $f: N_{1} \to L$ be a left minimal $\add N'$-approximation. 
We only have to prove that the cokernel $C:= \Coker f$ is \textup{(}$\mod G$\textup{)}-stable. 

For this purpose, we take a finite dimensional  $G$-module $X$ and show that $X\otimes_{\kk}^{G} C \cong C^{\oplus \dim X}$. 
First we claim that $X\otimes_{\kk}^{G} f$ is a left $\add N'$-approximation. 
Indeed, let $g :X \otimes_{\kk}^{G} N_{1} \to K$ be a homomorphism of $A$-modules with $K \in \add N'$. 
Recall that  the pair $(X \otimes_{\kk}^{G}-, \tuD(X) \otimes_{\kk}^{G}- )$  is an adjoint pair  by Proposition \ref{2023010031837}. 
Let $g^{\circ}: N_{1} \to \tuD(X) \otimes_{\kk}^{G} K$ be the homomorphism corresponding to $g$ via the adjoint isomorphism. 
Since $N'$ is \textup{(}$\mod G$\textup{)}-stable, we have $\tuD(X) \otimes_{\kk}^{G} K \in \add N'$. 
It follows that there exists a homomorphism $h: L \to \tuD(X) \otimes_{\kk}^{G} K$ such  that $h\circ f = g^{\circ}$. 
Let ${}^{\circ}h: X \otimes_{\kk}^{G} L \to K$ be the homomorphism corresponding to $h$ via the adjoint isomorphism. 
Then by the naturality of adjoint isomorphisms, we have ${}^{\circ}h \circ (X \otimes^{G}_{\kk} f) = g$. 
This shows that $X \otimes^{G}_{\kk} f$ is a left $\add N'$-approximation as desired. 
\[
\begin{xymatrix}@C=50pt{
X \otimes_{\kk}^{G} N_{1} \ar[r]^{X\otimes^{G}_{\kk} f} \ar[dr]_{g} & X \otimes_{\kk}^{G} L, \ar@{-->}[d]^{{}^{\circ} h } \\
 & K
 }\end{xymatrix} \ \ 
 \begin{xymatrix}@C=50pt{
N_{1} \ar[r]^{ f} \ar[dr]_{g^{\circ}} &  L \ar@{-->}[d]^{h } \\
 & \tuD(X) \otimes_{\kk}^{G} K 
 }\end{xymatrix}
\]

For simplicity, we set $d= \dim X$. We fix an isomorphism $X \otimes_{\kk}^{G} N_{1} \cong N_{1}^{\oplus d}$. 
By the above claim and the assumption that $f$ is a minimal left $\add N'$-approximation, 
we have homomorphism $s :L^{\oplus d} \to X \otimes_{\kk}^{G} L, \ t: X \otimes_{\kk}^{G} L \to L^{\oplus d}$
that compatible with $f^{\oplus d}$ and $X\otimes_{\kk}^{G} f$. 
\[
\begin{xymatrix}@C=60pt{
N_{1}^{\oplus d} \ar[d]^{\cong} \ar[r]^{f^{\oplus d}} & L^{\oplus d} \ar@<-2pt>[d]_{s} \\
X\otimes_{\kk}^{G} N_{1} \ar[r]_{X \otimes_{\kk}^{G} f} & X \otimes_{\kk}^{G} L \ar@<-2pt>[u]_{t}
}\end{xymatrix}
\]
Since $f^{\oplus d}$ is left minimal, the composition $t\circ s$ is an isomorphism. 
Moreover, as $X \otimes_{\kk}^{G} L$ and $L^{\oplus d}$ are  (non-canonically) isomorphic to each other and in particular 
have the same dimension, it follows that the morphisms $s$ and $t$ are isomorphisms. 
Now we conclude that $X \otimes_{\kk}^{G} C = \Coker(X \otimes_{\kk}^{G} f) \cong \Coker(f^{\oplus d}) \cong C^{\oplus d}$ 
as desired. 
\end{proof}

\begin{proof}[Proof of Theorem \ref{202309301357}]
Let $\sfF \subset \sfE$ be a complete set of representatives of $\sfE/G$.  
It follows from  Lemma \ref{202309301224}
that  $A^{\basic} \cong \bigoplus_{e \in \sfF} Ae$ and that it is  ind-\textup{(}$\mod G$\textup{)}-stable. 

Let $\mu^{-\NN}_{\bullet}(A)$ be the set of support $\tau$-tilting modules that are obtained from $A^{\basic}$ by  iterated left mutations. 
Applying Lemma \ref{202309301439} repeatedly, we see that 
$\mu^{-\NN}_{\bullet}(A) \subset (\sttilt A)^{\mod G}$. 
Since we are assuming 
$|(\sttilt \Lambda)^{G}|< \infty$, 
the set $(\sttilt A)^{\mod G} $ is a finite set 
by Theorem \ref{202309301601}.
It follows that the set $\mu^{-\NN}_{\bullet}(A)$ is also finite. 
Thus, by \cite[Corollary 2.38]{AIR}, we have $\mu^{-\NN}_{\bullet}(A) = \sttilt A$. 
Consequently we have $(\sttilt A)^{\mod G} =\sttilt A$. 
\end{proof}

\subsection{Complements}

\subsubsection{}\label{202407071658}

The aim of Section \ref{202407071658} is to show that if $(M,P)$ is a support $\tau$-tilting pair over $\Lambda$ such that $M$ is $G$-stable, 
then the pair $(F(M), F(P))$ is a support $\tau$-tilting pair over $A$ such that $F(M)$ is \textup{(}$\mod G$\textup{)}-stable.
Moreover, we study a converse problem. 
As is shown in  Example \ref{202310111303} below, 
even if a support $\tau$-tilting $\Lambda$-module $M$ is not $G$-stable, 
the $A$-module $F(M)$ possibly becomes a support $\tau$-tilting $A$-module.
However, the next proposition shows that a support $\tau$-tilting $\Lambda$-modules with the complement $P$ 
is $G$-stable if and only if the pair $(F(M), F(P))$ is a support $\tau$-tilting pair over $A$.

\begin{proposition}\label{202310061821}

Let $(M, P)$ be a support $\tau$-tilting pair for $\Lambda$. 
Then the following conditions are equivalent. 

\begin{enumerate}[(a)]

\item $M$ is $G$-stable. 

\item The pair $(F(M), F(P))$ is a support $\tau$-tilting pair for $A$.
\end{enumerate}

\end{proposition}

Via the reinterpretation given in Appendix \ref{202404011608}, 
this proposition and the proof of  Proposition \ref{202403121843} gives a generalization of \cite[Theorem 4.3]{BMM}.

We need the following lemma. 

\begin{lemma}\label{202310061811}
For a $\Lambda$-module $M$, the following conditions are equivalent. 

\begin{enumerate}[(1)]

\item $F(M)$ is a support $\tau$-tilting $A$-module.

\item $RF(M) = \bigoplus_{g \in G} {}_{g}M$ is a support $\tau$-tilting $\Lambda$-module. 

\end{enumerate}
\end{lemma}

\begin{proof}
(1)$\Rightarrow$(2). 
By Corollary \ref{202309151718}, $F(M)$ is \textup{(}$\mod G$\textup{)}-stable. 
The implication follows from  Theorem \ref{202309151712}. 

(2)$\Rightarrow$(1). 
By Lemma \ref{202309151023},  $RF(M)$ is $G$-stable. 
It follows from  Theorem \ref{202309151712} that  $FRF(M)$ is support $\tau$-tilting $A$-module. 
Since $\add FRF(M) = \add F(M)$ by Lemma \ref{202309151023}, we conclude that $F(M)$ is support $\tau$-tilting $A$-module. 
\end{proof}

\begin{proof}[Proof of Proposition \ref{202310061821}]
(a)$\Rightarrow$(b) 
By Theorem \ref{202309151712}, $F(M)$ is support $\tau$-tilting $A$-module. 
Let $Q$ be a projective $A$-modules such that $(F(M), Q)$ is a support $\tau$-tilting pair.

It follows from the adjoint isomorphism  $\Hom_{\Lambda}(R(Q), M) \cong \Hom_{A}(Q,F(M)) = 0$ 
that  $R(Q) \in \add P$ and hence $FR(Q) \in \add F(P)$. 
Since $Q$ is projective, the canonical surjection $FR(Q) = A \otimes_{\Lambda} Q \to Q$ splits. 
Thus we have $Q \in \add F(P)$. 

We recall the adjoint isomorphism 
\begin{equation}\label{202310061840}
\Hom_{A}(F(P), F(M) ) \cong \Hom_{ \Lambda}(P, RF(M)) \cong \Hom_{\Lambda}(P, \bigoplus_{g\in G}{}_{g}M).
\end{equation}
Now by the assumption that $M$ is $G$-stable, we have $\add (\bigoplus_{g\in G}{}_{g}M) = \add M$. 
It follows from \eqref{202310061840} that $\Hom_{A}(F(P), F(M)) = 0$ and  $F(P) \in \add Q$. 
Thus we conclude that $\add Q = \add F(P)$.

(b)$\Rightarrow$(a). 
The module $\bigoplus_{g \in G} {}_{g}M$ is support $\tau$-tilting $\Lambda$-module 
by Lemma \ref{202310061811}.  
In particular it is $\tau$-rigid. 
It follows that for any $g\in G$, we have 
\[
\Hom_{\Lambda}(M, \tau ({}_{g}M) ) = 0, \ \Hom_{\Lambda}({}_{g}M, \tau M) = 0.
\]
Moreover, since $\Hom_{A}(F(P), F(M) ) = 0$, it follows from \eqref{202310061840} 
that $\Hom_{\Lambda}(P, {}_{g}M) = 0$ for any $g \in G$. 
Thanks to \cite[Corollary 2.13]{AIR} we see that ${}_{g} M \in \add M$. 
Now it is easy to deduce that ${}_{g} M \cong M$. 
\end{proof}

We provide an example of $\Lambda$-module $M$ which is not a support $\tau$-tilting module, but $F(M)$ is 
a support $\tau$-tilting $A$-module.

\begin{example}\label{202310101649}
Let $\Lambda:= \kk Q/(ab, ba)$ where $Q$ is the following quiver  
\[
Q : \ 
\begin{xymatrix}{
1 \ar@<2pt>[r]^{a} & 2 \ar@<2pt>[l]^{b}. 
}\end{xymatrix}
\]
Let $G = \agl{g}$ be a cyclic group of order $2$ with a generator $g$. 
We give an action $G \curvearrowright \Lambda$ by 
$
g: \  1 \leftrightarrow 2, \ \  a \leftrightarrow b. 
$
Then there is an isomorphism $A := \Lambda * G \xrightarrow{\cong} M_{2}(\kk[x]/(x^{2}))$  
of algebras under which the element $e_{1}  \in A$ corresponds to the matrix unit $\begin{pmatrix} 1 & 0 \\ 0& 0 \end{pmatrix}$.

The projective $\Lambda$-module $P_{1} = \Lambda e_{1}$ is not a support $\tau$-tilting $\Lambda$-module, 
but $F(P_{1}) = A e_{1}$ is a support $\tau$-tilting $A$-module. 
\end{example}

We provide an example of a support $\tau$-tilting $\Lambda$-module $M$ which is not $G$-stable, but $F(M)$ is a support $\tau$-tilting $A$-module. 

\begin{example}\label{202310111303}
Let $\Lambda := \kk \times \kk$. We set $e_{1} :=(1_{\kk}, 0), \ e_{2} :=(0, 1_{\kk})$. 
Let $G:= \agl{g }$ be a cyclic group of order $2$ with a generator $g$. 
We give an action $G \curvearrowright \Lambda$ by setting $g(e_{1}) := e_{2}, \ g(e_{2}) := e_{1}$.
The skew group algebra $A:= \Lambda *G$ is isomorphic to the matrix algebra $M_{2}(\kk)$ of order $2$. 

%
%
%
%

It is easy to check that the simple $\Lambda$-module $S_{1}=\Lambda e_{1}$ corresponding to the vertex $1$ is a support $\tau$-tilting $\Lambda$ module  which is not $G$-stable, 
but  $F(S_{1})=A (e_{1} * e_{G})$ is a support $\tau$-tilting $A$-module.  
\end{example}

\subsubsection{}

In this section, we show that if $(N,Q)$ is a support $\tau$-tilting pair over $A$ such that $N$ is \textup{(}$\mod G$\textup{)}-stable, 
then the pair $(R(N), R(Q))$ is a support $\tau$-tilting pair over $\Lambda$ such that $R(N)$ is $G$-stable.
We  also study related problems as in the previous section. 

First we establish the following lemma which is an $R=\Res_{\Lambda}^{A}$-version of Lemma \ref{202310061811}.

\begin{lemma}\label{202310081547}
Let $N$ be  an $A$-module. 
The following conditions are equivalent
\begin{enumerate}[(1)]

\item $R(N)$ is a support $\tau$-tilting $\Lambda$-module.

\item $FR(N)  = \kk G \otimes_{\kk}^{G} N$ is a support $\tau$-tilting $A$-module. 

\end{enumerate}

\end{lemma}

\begin{proof}
(1) $\Rightarrow$ (2). 
By Corollary \ref{202309151718}, $R(N)$ is $G$-stable. 
Thanks to Theorem \ref{202309151712}, $FR(N)  = \kk G \otimes_{\kk}^{G} N$ is a support $\tau$-tilting $A$-module.

(2) $\Rightarrow$ (1).  
By Corollary \ref{202309151718},  $FR(N)$ is \textup{(}$\mod G$\textup{)}-stable. 
It follows from  Theorem \ref{202309151712} that  $RFR(N)$ is a support $\tau$-tilting $\Lambda$-module. 
By Corollary \ref{202309151718}, we have $\add RFR(N) = \add R(N)$. 
Thus, we conclude that $R(N)$ is a support $\tau$-tilting $A$-module. 
\end{proof}

The following is an $R=\Res^{A}_{\Lambda}$-version of Proposition \ref{202310061821}. 
However, the two conditions (a) and (b) are not equivalent in this situation.

\begin{proposition}\label{202310081630}

Let $(N,Q)$ be a support $\tau$-tilting pair over $A$. 
 We consider the following conditions.

\begin{enumerate}[(a)]

\item $N$ is \textup{(}$\mod G$\textup{)}-stable.

\item $(R(N), R(Q))$ is a support $\tau$-tilting pair over $\Lambda$.

\end{enumerate}
Then the following assertions hold.

\begin{enumerate}[(1)]

\item The condition (a) implies (b). 

\item If one of the following conditions hold, then the condition (b) implies (a).

\begin{enumerate}[(i)]

\item $\chara \kk$ does not divide $|G|$. 

\item The group algebra $\kk G$ is local. 

\end{enumerate}
\end{enumerate}

\end{proposition}

Via the reinterpretation given in Appendix \ref{202404011608}, 
this proposition and the proof of  Proposition \ref{202403121843} gives a generalization of \cite[Theorem 4.5]{BMM}.

\begin{proof}
(1) 
Assume that $N$ is \textup{(}$\mod G$\textup{)}-stable. By Theorem \ref{202309151712}, $R(N)$ is a support $\tau$-tilting $\Lambda$-module. 
Let $P$ be a projective $\Lambda$-module such that the pair $(R(N), P)$ is a support $\tau$-tilting pair for $\Lambda$. 
Then we have $\Hom_{A}(F(P), N ) \cong \Hom_{\Lambda}(P, R(N) ) =0$. It follows that $F(P) \in \add Q$ 
and hence $RF(P) \in \add R(Q)$. 
By Lemma \ref{202309151023}, $P$ is a direct summand of $RF(P)$, we have $P \in \add R(Q)$.

Since $N$ is \textup{(}$\mod G$\textup{)}-stable, we have $FR(N) \in \add N$. 
By the adjoint isomorphism, we have $\Hom_{\Lambda}(R(Q), R(N) ) \cong \Hom_{A}(Q, FR(N)) = 0$. 
It follows   that $R(Q) \in \add P$. Thus we conclude that $\add P = \add R(Q)$.

(2) 
Before assuming one of (i) or (ii) holds, we only assume that the condition (b) holds. 
Namely, 
we assume that $(R(N), R(Q))$ is a support $\tau$-tilting pair for $\Lambda$.  
Since $R(N)$ is $G$-stable, the pair $(FR(N), FR(Q))$ is a support $\tau$-tilting pair for $A$ by Proposition \ref{202310061821}. 
We claim that $\add FR(Q) = \add Q$ i.e., $Q$ is \textup{(}$\mod G$\textup{)}-stable. 
Indeed, since $Q$ is projective, the canonical surjection $FR(Q) =\kk G\otimes_{\kk}^{G} Q \to Q$ splits. 
Hence we have $Q \in \add FR(Q)$. 
On the other hand, we have $\Hom_{A}(FR(Q), N) \cong \Hom_{\Lambda}(R(Q), R(N) ) = 0$. 
It follows that $FR(Q) \in \add Q$. Thus we conclude that $\add FR(Q) = \add Q$. 

By the claim we have $|Q| = |FR(Q)|$. 
On the other hand, by a basic property of support $\tau$-tilting modules, we have $|A| = |N| +|Q| = |FR(N)| +|FR(Q)|$ . 
 It follows that  $|N| = |FR(N)|$. 

If we additionally assume  that one of (i) or (ii) holds, then we have $\add N \subset \add FR(N)$ by Lemma \ref{202404031821}. 
Thus, we conclude  $\add N = \add FR(N)$ as desired.
\end{proof}

We provide an example that shows that  the condition  (b) does not always imply the condition (a).

\begin{example}\label{202310081706}
Assume that $\chara \kk =3$. Let $G= \frkS_{3}$ be the symmetric group of degree $3$. 
We set  $\Lambda : = \kk $ with the trivial $G$ action. Then 
$A=\Lambda * G$ is the group algebra $\kk G$ and isomorphic to $\kk Q/(aba, bab)$ 
where $Q$ is a quiver given below 
\[
Q : \ 
\begin{xymatrix}{
1 \ar@<2pt>[r]^{a} & 2 \ar@<2pt>[l]^{b}. 
}\end{xymatrix}
\]
Note that we may identify  the simple modules $S_{1}, S_{2}$ corresponding to the vertices  $1,2$ with the 
trivial representation $\kk_{G}$ and the sign representation $\kk_{\sgn}$. 

Since $\kk G \otimes_{\kk}^{G} S_{2} \cong A$ as $A$-modules, 
the module   $N := S_{2} \oplus P_{2}$ is not \textup{(}$\mod G$\textup{)}-stable. 
But the pair $(N,0)$ is a supprot $\tau$-tilting pair for $A$ 
such that the pair $(R(N), R(0))$ is a support $\tau$-tilting pair for $\kk$. 
\end{example}

\section{Silting theory of skew group algebra extensions}\label{202404102225}

In this section we discuss silting theory of a skew group algebra extension $\Lambda \subset \Lambda *G$. 
For the definition of silting complexes and other basic notions and results, 
we refer \cite{AI}.

\subsection{Preliminaries}

Recall that  the induction functor $F=\Ind_{\Lambda}^{A}: \mod \Lambda \to \mod A$ and the restriction functor $R=\Res_{\Lambda}^{A}: \mod A \to \mod \Lambda$ are exact and preserve projective modules. 
Therefore, these functors $F,R$ induce
exact functors $F =\Ind_{\Lambda}^{A}: \Dbmod{\Lambda} \to \Dbmod{A}$ 
 and $R=\Res_{\Lambda}^{A} : \Dbmod{A} \to \Dbmod{\Lambda}$ of derived categories, 
 which are denoted by the same symbols, 
 and 
the functors $F$ and $R$ on the derived categories  restrict 
to exact functors  $F =\Ind_{\Lambda}^{A}: \Kbproj{\Lambda} \to \Kbproj{A}$ 
 and $R=\Res_{\Lambda}^{A} : \Kbproj{A} \to \Kbproj{\Lambda}$ of homotopy categories, 
 which are also denoted by the same symbols.

 Let $g \in G$. 
 Then the endofunctor ${}_{g}(-): \mod \Lambda \to \mod \Lambda, \ M \mapsto {}_{g}M$ is exact and preserves finitely generated projective modules. 
 Thus it gives rise to an exact endofunctor ${}_{g}(-)$ of the derived category $\Dbmod{\Lambda}$ and the homotopy category 
 $\Kbproj{\Lambda}$. 
 We define the $G$-stability of an object $M \in \Kbproj{\Lambda}$ 
 in a similar way of Definition \ref{202309151111}. 
 
 Let $X$ be a finite dimensional $G$-module. 
 Then the endofunctor $X \otimes_{\kk}^{G}-: \mod A \to \mod A$ is exact and preserves finitely generated projective modules by Corollary \ref{202310021728}. 
 Thus it gives rise to an exact endofunctor $X \otimes_{\kk}^{G} -$ of the derived category $\Dbmod{A}$ and the homotopy category 
 $\Kbproj{A}$. 
 We define the \textup{(}$\mod G$\textup{)}-stability of an object $N \in \Kbproj{A}$ 
 in a similar way of Definition \ref{202309171251}.

Moreover, all isomorphisms established in Section \ref{202403291602} and Section \ref{202403291603} 
are natural in relevant modules over $\Lambda$ or $A$, 
it follows that all the results given in these sections have their derived category version. 
Among other things we point out the following derived category version of Proposition \ref{202309150913} 
for later quotations. 
To state this, we recall that an object $N \in \Dbmod{A}$ is called \emph{rigid} if $\Hom_{\Dbmod{A}}(N,N[1]) = 0$.

\begin{proposition}\label{2023091509131}
For a rigid object $N \in\Dbmod{A}$, the following conditions are equivalent. 

\begin{enumerate}[(1)]

\item $N$ is \textup{(}$\mod G$\textup{)}-stable.

\item $S \otimes_{\kk}^{G} N \in \add N$ for any simple $G$-modules $S$.

\item $\add FR(N) = \add N$.

Moreover, if we further assume  one of the following conditions: (i) $\chara \kk \nmid |G|$, (ii) $\kk G$ is local, 
then the above conditions are equivalent to the following condition (4). 

\hspace{-38pt} (4) $FR(N) \in \add N$. 
\end{enumerate}

\end{proposition}

\subsection{The bijection}
The aim of this section is to 
prove that  the induction-restriction adjoint pair $(F,R)$ induces a bijection between $G$-stable silting objects and \textup{(}$\mod G$\textup{)}-stable silting complexes, 
whose restriction maps  give bijections of the sets of $2$-term silting complexes having appropriate stabilities, and the set of tilting objects having appropriate stabilities.   
For this purpose, first we prove the following  lemma.

\begin{lemma}\label{202310012156}
\begin{enumerate}[(1)]

\item 

If $S\in \sfK^{\mrb}(\Lambda\proj)$ is a $G$-stable tilting (resp. silting, $2$-term silting) object of $\sfK^{\mrb}(\Lambda\proj)$, 
then 
$F(S) \in \sfK^{\mrb}(A\proj)$ is a tilting (resp. silting, $2$-term silting) object of $\sfK^{\mrb}(A\proj)$.

\item 
If $T\in \sfK^{\mrb}(A\proj)$ is a \textup{(}$\mod G$\textup{)}-stable tilting (resp. silting, $2$-term silting) object of $\sfK^{\mrb}(A\proj)$, 
then 
$R(T) \in \sfK^{\mrb}(\Lambda\proj)$ is a tilting (resp. silting, $2$-term silting) object of $\sfK^{\mrb}(\Lambda\proj)$. 

\end{enumerate}
\end{lemma}

\begin{proof}
(1) 
Since $\Lambda \in \sfK^{\mrb}(\Lambda\proj) = \thick S$, we have $A=F(\Lambda) \in \thick F(S)$. 
It follows that  $\thick F(S)=\sfK^{\mrb}(A\proj)$.  

For $n \in \ZZ$, we have
\begin{equation}\label{202403301545}
\Hom_{\sfK^{\mrb}(A\proj)}(F(S),F(S)[n]) \cong 
\Hom_{\sfK^{\mrb}(\Lambda \proj)}(S,RF(S)[n]).  
\end{equation}
Since $S$ is $G$-stable, we have $RF(S) \in \add S$ by (a derived category version of) Corollary \ref{202310021718}. 
It follows from the assumption that  $S$ is tilting (resp. silting) that 
 $\Hom_{\sfK^{\mrb}(A\proj)}(F(S),F(S)[n]) = 0$ for $n\neq 0$ (resp. $n >0$).
Now we have shown that if $S$ is a tilting (resp. silting) object, then so is $F(S)$.

Finally, since the functor $F$ preserves projective modules, it follows that if $S$ is a $2$-term silting complex, then so is $F(S)$.

(2) can be  proved in a similar argument by using  Proposition \ref{2023091509131} 
and the fact that $R(A)$ contains $\Lambda$ as a direct summand. 
%
\end{proof}

Let $\silt \Lambda$ be the poset (of isomorphism classes) of basic silting complexes over $\Lambda$.
We denote by $(\silt \Lambda)^{G}$ the subposet of $G$-stable basic silting complexes. 
We also introduce  
 $\2silt \Lambda$ and $(\2silt \Lambda)^{G}$  (resp. $\tilt \Lambda$ and $(\tilt \Lambda)^{G}$) the poset of $2$-term silting complexes 
 (resp. tilting complexes) and the subposet of $G$-stable ones.

 In a similar way, we define $(\silt A)^{\mod G}, (\2silt A)^{\mod G}, (\tilt A)^{\mod G}$ 
 the poset of  \textup{(}$\mod G$\textup{)}-stable silting complexes, $2$-term silting complexes and tilting complexes.

 We remark that since we are not assuming that our algebras $\Lambda, A$ are  basic,  strictly speaking 
 the silting complex $\Lambda$ (resp. $A$) does not necessarily belong to $\silt \Lambda$ (resp. $\silt A$). 
 However, to simplify the notation, we often denote $\Lambda$ and $A$ to denote the basic silting complexes $\Lambda^{\basic}$ 
 and $A^{\basic}$.

\begin{theorem}\label{202403121806}
The adjoint pair $(F=\Ind_{\Lambda}^{A},R=\Res_{\Lambda}^{A})$ induces  isomorphisms of posets
\[
F=\Ind_{\Lambda}^{A}: (\silt \Lambda)^{G} \longleftrightarrow (\silt A)^{\mod G}: \Res_{\Lambda}^{A} =R. 
\]
Moreover, the restrictions of these maps give isomorphisms  
$F: (\2silt \Lambda)^{G} \leftrightarrow (\2silt A)^{\mod G}:R$ and $ F: (\tilt \Lambda)^{G} \leftrightarrow (\tilt A)^{\mod G}:R$. 
\end{theorem}

\begin{proof}
Thanks to Lemma \ref{202310012156} the map 
$F:  (\silt \Lambda)^{G} \to (\silt A)^{\mod G}, S \mapsto F(S)^{\basic}$ and 
$R: (\silt A)^{G} \to (\silt \Lambda)^{G}, T \mapsto R(T)^{\basic}$ are well defined. 
By (a derived category version of) Corollary \ref{202310021718} and Proposition \ref{2023091509131}, 
we see that these maps are inverse to each other. 
It  also follows from Lemma \ref{202310012156} that restrictions of these maps 
give bijections $F: (\2silt \Lambda)^{G} \leftrightarrow (\2silt A)^{\mod G}:R$ and $ F: (\tilt \Lambda)^{G} \leftrightarrow (\tilt A)^{\mod G}:R$.

Recall that for  two silting complexes $S_{1}, S_{2}$, the inequality $S_{1} \geq S_{2}$ is defined by the 
equality $\Hom(S_{1}, S_{2}[n]) = 0  \ ( \forall n > 0)$. 
Using adjoint isomorphisms as in \eqref{202403301545} and (a derived category version of) Corollary \ref{202310021718} and Proposition \ref{2023091509131}, 
we can check that the above maps are poset morphisms. 
\end{proof}

\subsection{Correspondence between $2$-term silting complexes and support $\tau$-tilting modules}\label{202403121920}

Recall that taking the $0$-th chomology group $\tuH^{0}(S)$ of a $2$-term silting complex $S$ gives a bijection 
$\tuH^{0}: \2silt \Lambda \to \sttilt \Lambda$.

\begin{proposition}\label{202403121843}

\begin{enumerate}[(1)]

\item 
The  bijection 
$\tuH^{0}: \2silt \Lambda \to \sttilt \Lambda$ restricts  to  a bijection 
$\tuH^{0}: (\2silt \Lambda)^{G} \to (\sttilt \Lambda)^{G}$.

\item 
The  bijection 
$\tuH^{0}: \2silt A \to \sttilt A$  restricts to a bijection 
$\tuH^{0}: (\2silt A)^{\mod G} \to (\sttilt A)^{\mod G}$.

\end{enumerate}

\end{proposition}

\begin{proof}
(1) 
We denote the inverse map of $\tuH^{0}:\2silt \Lambda \to \sttilt \Lambda$ by $\tuK: \sttilt \Lambda \to \2silt\Lambda$. 
It is clear that if $S$ is a $G$-stable $2$-term silting complex, then $\tuH^{0}(S)$ is $G$-stable. 
Thus, we only have to show that if $M$ is a $G$-stable support $\tau$-tilting module, then the $2$-term silting complex $\tuK(M)$ is $G$-stable. 

Recall $\tuK(M)$ is constructed in the following way. 
Let $P$ be a complement projective module of $M$ and $P^{-1} \to P^{0} \to M \to 0$ a minimal projective presentation.   
For simplicity we set $P_{M}:=  [P^{-1} \to P^{0}]$. 
Then  $\tuK(M) := P_{M}  \oplus P[1]$.

By Corollary \ref{202309180059}, the complex $RF(P_{M})$ gives  a minimal projective  presentation of $RF(M)$. 
Since $\add RF(M) = \add M$, we have $\add RF(P_{M})=\add P_{M}$ in $\sfK^{\mrb}(\Lambda\proj)$. 
Next note that $RF(P)$ is projective $\Lambda$-module. 
Since the pairs $(R,F)$ and $(F,R)$ are adjoint pairs, we have $\Hom_{\Lambda}(RF(P), M) \cong \Hom_{\Lambda}(P,RF(M))=0$. 
It follows $RF(P) \in \add P$. 
On the other hand, we have $P \in \add RF(P)$ by Lemma \ref{202309151023}. 
It follows that $\add RF(P) = \add P$. 
Combining above two observations we conclude that $\add RF \tuK(M) = \add \tuK(M)$.

(2) is proved in a similar way by using Proposition \ref{2023091512311}(2) and the fact that the canonical morphism $RF(Q) \to Q$ is surjective.
\end{proof}

\begin{remark}
The equality $\add RF(P) = \add P$ in the above proof immediately follows from 
 Propositions \ref{202310061821} and \ref{202310081630}. 
 However, these two propositions  rely on Theorem \ref{202309151712}. 
Therefore, if we use these two propositions in the proof, we can not use Proposition \ref{202403121843} to the second proof of the theorem given below. 
\end{remark}

\begin{proof}[The second proof of Theorem \ref{202309151712}]
(1) Let $M$ be a $G$-stable support $\tau$-tilting $\Lambda$-module. 
Then $\tuH^{0}F\tuK(M)$ is a \textup{(}$\mod G$\textup{)}-stable support $\tau$-tilting $A$-module 
by Theorem \ref{202403121806} and Proposition \ref{202403121843}. 
\[
(\sttilt \Lambda)^{G} \xrightarrow{ \ \tuK \ } (\2silt \Lambda)^{G} \xrightarrow{ \ F \ } (\2silt A)^{\mod G} \xrightarrow{ \ \tuH^{0} \ } (\sttilt A)^{\mod G}.
\]
On the other hand, it is straightforward to check that $\tuH^{0}F\tuK(M) = F(M)$. 
This finishes the proof of (1). 
The second statement is proved in a similar way.
\end{proof}

\subsection{Silting discreteness}

We  introduce the notion of $G$-stable silting discreteness. 

\begin{definition}\label{202403311837}
An algebra $\Lambda$ with a $G$-action  is called 
\emph{$G$-stable silting discrete} 
if for any $n \geq 0$, there are only a finite number of $G$-stable silting complexes $S$ over $\Lambda$ 
such that $\Lambda[n]\leq S \leq \Lambda$. 
In other words, we have $| (\silt \Lambda)^{G} \cap [\Lambda[n],\Lambda]|< \infty$ 
where 
\[
[\Lambda[n], \Lambda]:= \{ S\in \silt \Lambda\mid \Lambda[n] \leq S \leq \Lambda\}.
\]
\end{definition}

We provide a silting version of Theorem \ref{202309301357}. 

\begin{theorem}\label{202403311841}
Let $G$ be a finite  group and  $\Lambda$  an algebra with $G$-action that preserves idempotent elements and has  local stabilizers. 
Assume that $\Lambda$  is $G$-stable silting discrete. 
 Then the following assertions hold.
 
 \begin{enumerate}[(1)]
 
\item  $A = \Lambda * G$ is silting discrete.

\item All silting complexes over  $A$ are \textup{(}$\mod G$\textup{)}-stable.

\item The adjoint pair $(\Ind_{\Lambda}^{A}, \Res_{\Lambda}^{A})$ induces a bijection 
\[
\Ind_{\Lambda}^{A}: ( \silt \Lambda)^{G} \longleftrightarrow \silt A: \Res_{\Lambda}^{A}.
\]
\end{enumerate}

\end{theorem}

\begin{proof}
We can introduce the notion of ind-\textup{(}$\mod G$\textup{)}-stability of objects of $\Dbmod{\Lambda}$
as in the same way of Definition \ref{202403311650}. 
Then we can prove by the same method of Lemma \ref{202309301439}
 that 
every irreducible  left or  irreducible right silting   mutation $\mu^{\silt, \pm}_{i}(S)$ of  ind-\textup{(}$\mod G$\textup{)}-stable silting complexes over $A$ 
is again ind-\textup{(}$\mod G$\textup{)}-stable.  
On the other hand, as is mentioned in the proof of Theorem \ref{202309301357}, the $A$-module $A^{\basic}$ is ind-\textup{(}$\mod G$\textup{)}-stable. 
Therefore, if we denote by $\mu^{\silt, \ZZ}_{\bullet}(A)$ the set of all silting complexes over $A$ obtained by silting mutations from $A^{\basic}$, 
then $\mu^{\silt,\ZZ}_{\bullet}(A) \subset (\sttilt A)^{\mod G}$. 
%

Thus we only have to show that every silting complex $T$ over $A$ belongs to $\mu_{\bullet}^{\silt, \ZZ}(A)$. 
Shifting the cohomological grading, we may assume that $T \in [A[n],A]$ for some $n \geq 0$. 
Assume   that $T \notin \mu_{\bullet}^{\silt, \ZZ}(A)$. 
Then 
 there exists an infinite  sequence of irreducible left silting mutations  starting from $A$ that is bounded by $T$ and hence by $A[n]$.
\[
A^{\basic} \xrightarrow{ \ \mu_{i_{1}}^{\silt, -} \  }  T_{1}\xrightarrow{ \ \mu_{i_{2}}^{\silt, -} \  }  T_{2} \xrightarrow{ \ \mu_{i_{3}}^{\silt, -} \  } \cdots  > T > A^{\basic}[n]. 
\]
Thus  the set
$(\silt A)^{\mod G} \cap [A[n], A]$ contains an infinite set $\{A, T_{1}, T_{2}, \dots \}$. 
It follows from Theorem \ref{202403121806} that  $ | (\silt \Lambda)^{G} \cap [\Lambda[n],\Lambda]| =\infty$, a contradiction. 
\end{proof}

\subsection{The case that $\Lambda$ is Frobenius}\label{202404041518}

In this section we assume that $\Lambda$ is (finite dimensional) Frobenius.

\subsubsection{Preliminary}

We always equip a Frobenius algebra $\Lambda$ with a Nakayama automorphism $\nu_{\Lambda}$ 
and a Nakayama form $\agl{-,+}_{\Lambda}: \Lambda \times \Lambda \to \kk$, which now we recall.  
A Nakayama automorphism $\nu_{\Lambda}:\Lambda \to \Lambda$ is an algebra automorphism 
and a Nakayama form $\agl{-,+}_{\Lambda}: \Lambda \times\Lambda \to \kk$ is a non-degenerate bilinear form
 that satisfies the equations 
$\agl{rs, t}_{\Lambda}= \agl{r,st}_{\Lambda} =\agl{s, t\nu_{\Lambda}(r)}$ 
for any $r,s,t\in \Lambda$. 
We note that a Nakayama automorphism is unique up to inner automorphisms. 
We also note that the map $\phi: \Lambda \to \tuD(\Lambda), \ \phi_{r}(s) := \agl{r,s}_{\Lambda}$ 
gives an isomorphism $\phi: \Lambda \to {}_{\nu}\tuD(\Lambda)$ of $\Lambda$-$\Lambda$-bimodules. 

Following \cite[Definition 11]{AJS}, 
we say that a group action $G \curvearrowright \Lambda$ on a Frobenius algebra $\Lambda$ \emph{preserves the Nakayama form} 
if we have $\agl{g(r), g(s)}_{\Lambda} =\agl{r,s}_{\Lambda}$ for all $r,s \in \Lambda$ and $g \in G$.

\begin{proposition}[{cf. \cite[Proposition 4.4, Corollary 4.5]{AJS}}]\label{202404041529}
Let $\Lambda$ be a Frobenius algebra with an action of $G$ that preserves the Nakayama form. 
Then the skew group algebra $A=\Lambda * G$ is a Frobenius algebra 
with a Nakayama automorphism $\nu_{A}$ given by 
\[
\nu_{A}(r* g) := \nu_{\Lambda}(r) *g \ \ ( r\in \Lambda, \ g \in G).
\]
\end{proposition} 

\begin{proof}
For simplicity we write $\nu =\nu_{\Lambda}, \ \agl{-,+} = \agl{-,+}_{\Lambda}$. 
First we claim that $\nu g(r) = g \nu (r)$ for all $r\in \Lambda$. 
Indeed we can verify, as below, that the equation $\agl{ \nu^{-1}g^{-1}\nu g (r), s} = \agl{r,s}$ holds for any $s\in \Lambda$:
\[
\begin{split}
\agl{\nu^{-1}g^{-1} \nu g(r), s} =\agl{s, g^{-1}\nu g(r)} = \agl{g(s), \nu g(r)} = \agl{g(r), g(s)} = \agl{r,s}
\end{split}
\]

Now it is straightforward to check that the bilinear form 
$\agl{-,+}_{A}: A \times A \to \kk$ 
given by 
\[
\agl{r*g, s*h}_{A} := \delta_{g,h^{-1}}\agl{r,g(s)} \ \ (r,s \in \Lambda,  \ g,h \in G)
\]
is a Nakayama form whose Nakayama automorphism is $\nu_{A}$. 
%
%
\end{proof} 

%
%
%
%

\subsubsection{Twist by Nakayama automorphism of finite order}

Let $\Lambda$ be a Frobenius algebra.  
Now we  assume that $\Lambda$ has a Nakayama automorphism $\nu$ of order $N< \infty$ in the group   $\Aut \Lambda$ of algebra automorphisms of $\Lambda$. 
Let $\agl{\nu}\cong \ZZ/N\ZZ$ be the subgroup of $\Aut \Lambda$ generated by $\nu$. 
We consider the canonical  action $\agl{\nu} \curvearrowright \Lambda$ to be $\nu^{n} \cdot r := \nu^{n}(r)$ for $r\in \Lambda$ and $n \in \ZZ/N\ZZ$.

We denote  the posets of  $\agl{\nu}$-stable objects 
$(\sttilt \Lambda)^{\nu} := (\sttilt \Lambda)^{\agl{\nu}},$ \ $(\silt \Lambda)^{\nu}:= (\silt \Lambda)^{\agl{\nu}}$ etc. 

We point out the  lemma below that follows from \cite[Theorem A.4.]{Aihara}, \cite[Propositions 2.5, 2.6]{Adachi-Kase}. 

\begin{lemma}\label{202405081445}
We have the following equalities  and an isomorphism of the posets
\[
\twotilt \Lambda = (\2silt \Lambda)^{\nu}\cong (\sttilt \Lambda)^{\nu}, \ \tilt \Lambda = (\silt \Lambda)^{\nu}.\]
\end{lemma}

The following corollary plays a key role in applications given in the next section.

\begin{corollary}\label{202403311407}
In the above setup, the skew group algebra $A:= \Lambda * \agl{\nu}$ is symmetric 
and we have the following equalities and an isomorphism of posets
\[
\twotilt A=  \2silt A \cong \sttilt A , \  \tilt A=\silt A.
\]
\end{corollary}

\begin{proof}
By Proposition \ref{202404041529},  
the automorphism $\nu_{A}$ of $A$ given by $\nu_{A}(r*\nu^{m}) := \nu_{\Lambda}(r) * \nu^{m}$ is a Nakayama automorphism of $A$. 
It is straightforward to check that $\nu_{A}$ is the inner automorphism $\ad_{\epsilon}=\epsilon \cdot - \cdot \epsilon^{-1}$ associated to the element $\epsilon := 1_{\Lambda} *\nu$. 
It follows that $A$ is symmetric. 
Applying  Lemma \ref{202405081445} to the algebra $A$, we obtain the desired equalities of posets. 
\end{proof}

%

Applying  Theorem \ref{202309301601}, Theorem \ref{202403121806} and Proposition \ref{202403121843}, 
we obtain the following result. 

\begin{theorem}\label{202405081452}
The adjoint pair $(\Ind_{\Lambda}^{A},\Res^{A}_{\Lambda})$ induces the following isomorphisms of posets 
\[
\begin{split}
&\twotilt \Lambda \cong (\twotilt A)^{\agl{\nu}\mod} \cong (\sttilt A)^{\agl{\nu}\mod}, \ \ \ \tilt \Lambda \cong (\silt A)^{\agl{\nu}\mod}. 
\end{split}
\]
\end{theorem}

Finally, 
applying Theorem \ref{202309301357} and Theorem \ref{202403311841}, we obtain the following result. 

\begin{theorem}\label{202405081500}
Assume that the action $\agl{ \nu} \curvearrowright \Lambda$ preserves idempotent elements and has local stabilizers and that 
$\Lambda$ is tilting discrete. 
Then the  adjoint pair $(\Ind_{\Lambda}^{A},\Res^{A}_{\Lambda})$ induces the following isomorphisms of posets 
\[
\begin{split}
&\twotilt \Lambda \cong \twotilt A \cong \sttilt A, \ \ \ \tilt \Lambda \cong \silt A. 
\end{split}
\]
\end{theorem}

\section{Application to preprojective algebras and folded mesh algebras}\label{202404102228}

In this section, applying results obtained in the previous sections, 
we study silting theory and $\tau$-tilting theory of preprojective algebras and folded mesh algebras. 

\subsection{The preprojective algebra $\Pi(\LL_{n})$ of type $\LL_{n}$}

In this section, 
we  determine  the posets $\sttilt \Pi(\LL_{n}), \ \silt \Pi(\LL_{n})$  in terms of the Weyl group and the braid group. 

\subsubsection{}

Recall that the generalized Dynkin diagrams introduced by Happel-Preiser-Ringel \cite{HPR} 
consist of the usual Dynkin diagrams $\AA_{n}$, $\BB_{n}$, $\CC_{n}$, $\DD_{n}$, $\EE_{6}$, $\EE_{7}$, $\EE_{8}$, $\FF_{4}$, $\GG_{2}$ 
and  one additional case $\LL_{n}$ defined as 
\[\LL_{n} \ \ 
\begin{xymatrix}{
1 \ar@{-}[r] & 2 \ar@{-}[r] 
 & 3 \ar@{-}[r] & \cdots  \ar@{-}[r] 
  & n-1  \ar@{-}[r] 
  & n   \ar@{-}@(rd,ru).
 }\end{xymatrix} 
\]
We recall from \cite[Section 7]{ES}, the definition of the preprojective algebra $\Pi(\LL_{n})$ of type $\LL_{n}$.
The algebra $\Pi(\LL_{n})$ is given by the following quiver 
\[
\begin{xymatrix}{
1 \ar@<2pt>[r]^{\beta_{1}} & 2 \ar@<2pt>[r]^{\beta_{2}} \ar@<1pt>[l]^{\beta_{1}^{*}} 
 & 3 \ar@<2pt>[r]^{\beta_{3}} \ar@<1pt>[l]^{\beta_{2}^{*}} & \cdots  \ar@<2pt>[r]^{\beta_{n-2}} \ar@<1pt>[l]^{\beta_{3}^{*}} 
  & n-1  \ar@<2pt>[r]^{\beta_{n-1}} \ar@<1pt>[l]^{\beta_{n -2}^{*}} 
  & n  \ar@<1pt>[l]^{\beta_{n-1}^{*}}  \ar@(rd,ru)_{\gamma}
 }\end{xymatrix} 
\]
with the relation 
$\gamma^{2} + \sum_{i=1}^{n} (\beta_{i} \beta^{*}_{i} - \beta^{*}_{i-1}\beta_{i -1})$,  
where we set $\beta_{0}=\beta_{0}^*=\beta_{n}=\beta_{n}^*:= 0$. 
Representation theory of $\Pi(\LL_{n})$ and its deformations have been studied by several researchers (see e.g., \cite{BES,HZ}). 
The generalized Dynkin diagrams and their preprojective algebras also appear in the study of the quantum groups and tensor categories related to conformal field theory (see e.g., \cite{EK, KO, MOV}). We note that in this area the diagram $\LL_{n}$ is usually named as $T_{n}$.

The main theorem of this section tells that 
the posets $\sttilt \Pi(\LL_{n}), \ \silt \Pi(\LL_{n})$ are isomorphic to 
the Weyl group $W$ and the braid group $B$ of type $\BB_{n}$ and $\CC_{n}$ (equipped with the right weak order), which now we recall.

Recall that the underlying valued graphs $|\Delta_{\BB_{n}}|, |\Delta_{\CC_{n}}|$ of the Dynkin diagrams $ \Delta_{\BB_{n}}, \Delta_{\CC_{n}}$ of types 
$\BB_{n}, \CC_{n}$ coincide to each other. 
It follows that  the Weyl groups and the braid groups of types $\BB_{n}$ and $\CC_{n}$ coincide to each other. 
For simplicity we set   $|\Delta|:= |\Delta_{\BB_{n}}| = |\Delta_{\CC_{n}}|$.  
$$|\Delta|: 
\xymatrix{
1 \ar@{-}[r]^{4} & 2 \ar@{-}[r]  & \cdots &\ar@{-}[r] \cdots & n -1\ar@{-}[r]  & n. 
}$$

Recall that the \emph{Weyl group} $W:=W_{|\Delta|}$ is defined by the generators 
$s_i$ and relations 
$(s_is_j)^{m(i,j)}=1$, 
where  
\[
m(i,j):=\left\{\begin{array}{ll}
1\ \ \ \ \ &\mbox{if $i=j$,}\\
2\ \ \ \ \ &\mbox{if no edge between $i$ and $j$ in $\Delta$,}\\
3\ \ \ \ \ &\mbox{if there is an edge $i\stackrel{ }{\mbox{---}}j$ in $\Delta$,}\\
4\ \ \ \ &\mbox{if  there is an edge $i\stackrel{4}{\mbox{---}}j$ in $\Delta$.}\\
\end{array}\right.\]
(We refer to \cite{BB,H,KT} for the background of Weyl groups and braid groups.)
We regard $W$ as a poset defined by the right weak order. 

The braid group $B=B_{|\Delta|}$ is defined by generators $a_i$  and relations 
$(a_ia_j)^{m(i,j)}=1$ for $i\neq j$  (i.e. the difference with $W_{\Delta}$ is that we do not require the relations $a_i^2=1$ for any $i$). 
We regard $B$ as a poset defined by right-divisibility order.

Now we can state the main result of this section which gives a classification of (2-term) tilting complexes for preprojective algebra of type $\LL_n$.

\begin{theorem}\label{typeL}
Let $\Pi(\LL_{n})$ be the preprojective algebra of type $\LL_n$. 
\begin{itemize}
\item[(1)] We have the following  anti-isomorphism of posets
$$W\to   \sttilt \Pi(\LL_n).$$
\item[(2)] We have the following  anti-isomorphism of posets
$$B\to  \silt \Pi(\LL_n).$$
\end{itemize}
\end{theorem}

\subsubsection{Proof of  Theorem \ref{typeL}}

Recall that the preprojective algebra $\Pi(\AA_{n})$ of type $\AA_{n}$ is given by 
the following quiver 
\[
\begin{xymatrix}{
1 \ar@<2pt>[r]^{\alpha_{1}} & 2 \ar@<2pt>[r]^{\alpha_{2}} \ar@<1pt>[l]^{\alpha_{1}^{*}} 
 & 3 \ar@<2pt>[r]^{\alpha_{3}} \ar@<1pt>[l]^{\alpha_{2}^{*}} & \cdots  \ar@<2pt>[r]^{\alpha_{n-2}} \ar@<1pt>[l]^{\alpha_{3}^{*}} 
  & n-1  \ar@<2pt>[r]^{\alpha_{n-1}} \ar@<1pt>[l]^{\alpha_{n -2}^{*}} 
  & n  \ar@<1pt>[l]^{\alpha_{n-1}^{*}} 
 }\end{xymatrix} 
\]
with the mesh relation $\sum_{i=1}^{n} (\alpha_{i} \alpha^{*}_{i} - \alpha^{*}_{i-1}\alpha_{i-1})$, 
where we set $\alpha_{0}=\alpha_{0}^*=\alpha_{n}=\alpha_{n}^*:= 0$. 
By \cite[Definition 4.6, Theorem 4.8]{BBK}, the algebra automorphism  $\nu$ of $\Pi(\AA_{n})$  given below is a Nakayama automorphism of $\Pi(\AA_{n})$.  
\[
\nu: e_{i} \mapsto e_{n-i+1}, \  \alpha_{i} \mapsto \alpha^{*}_{n-i}, \  \alpha_{i}^{*} \mapsto \alpha_{n-i}.  
\] 
Observe  that  the map $\nu$ has order $2$. Therefore we can take the Nakayama twisted algebra $\Pi(\AA_{n}) * \agl{\nu}$.

To prove Theorem \ref{typeL},  we need the following observation, which seems to be  well-known for experts and proved by a straightforward computation.   
But for the convenience of the readers, we give a sketch of the proof.

\begin{lemma}[{cf. \cite[Section 7]{ES}}]\label{202405221519}
The following assertions hold.
\begin{enumerate}[(1)]
\item 
The skew group algebra $\Pi(\AA_{2n}) *\agl{\nu}$ is Morita equivalent to $\Pi(\LL_{n})$. 

\item 
The algebra $\Pi(\LL_{n})$ is symmetric. 
\end{enumerate}
\end{lemma}

\begin{proof}
(1) 
For simplicity we set $\Lambda := \Pi(\AA_{2n}), G := \agl{\nu}, A := \Lambda *G$.
The subset $\{ e_{1}, \dots, e_{n} \}$ of vertexes is a complete set of the representative of the quotient set $\{ e_{1}, \dots, e_{2n}\}/\agl{\nu}$. 
It follows from Lemma \ref{202309301224}  that if we set $\epsilon := \sum_{i=1}^{n}e_{i}$ then the algebra $A':= \epsilon A \epsilon$ is a basic algebra 
which is Morita equivalent to $A$. Thus we only have to prove  $A' \cong \Pi(\LL_{n})$. 

For $i=1,2,\dots, n-1$ we set $\beta_{i} := \alpha_{i}, \ \beta_{i}^{*} := \alpha^{*}_{i}$ 
and $\gamma := \alpha_{n} (e_{n+1}* \nu)$ where we regard these as elements of $A'$. 
It follows from the mesh relation of $\Lambda =\Pi(\AA_{2n})$ that for $i=1,2, \dots, n-1$ we have 
$\beta_{i}\beta_{i}^{*} -\beta^{*}_{i -1} \beta_{i-1} = 0$.
It also follows from the mesh relation of $\Lambda =\Pi(\AA_{2n})$ that 
 $\gamma^{2} = \alpha_{n}\alpha_{n}^{*} = \alpha^{*}_{n-1} \alpha_{n-1} = \beta^{*}_{n-1} \beta_{n -1}$.  
 Thus we have $\gamma^{2} -\beta^{*}_{n-1} \beta_{n-1}=0$. 

Now we obtain a canonical algebra homomorphism $f: \Pi(\LL_{n}) \to A'$.
It is straightforward to check that elements $\gamma, \beta_{i}, \beta^{*}_{i} \ ( i=1,2, \dots, n-1)$ are generators of the algebra $A'$ 
and hence that the map $f$ is surjective.  
Finally we can  compute the dimensions of both algebras and check that they are of the same dimension. 
Thus we conclude that the map $f$ is an isomorphism.

(2) follows from (1) and Corollary \ref{202403311407}.
\end{proof}

We  recall the following result. 

\begin{theorem}\label{tau-weyl}
\begin{itemize}
\item[(1)]We have the following  anti-isomorphism of posets
$$W \to \twotilt\Pi(\AA_{2n}).$$
\item[(2)]We have the following  anti-isomorphism of posets
$$B \to  \tilt\Pi(\AA_{2n}).$$
\end{itemize}
\end{theorem}

\begin{proof}
They follows from  \cite{Aihara-Mizuno,Mizuno: nu-stable}.
\end{proof}

Now we see that our main theorem is a simple consequence of a combination of  the above preparations and our general result developed in the previous sections.

\begin{proof}[Proof of Theorem \ref{typeL}]
By Lemma \ref{202405221519}, $\Pi(\AA_{2n})*\langle \nu\rangle$ is Morita equivalent to $\Pi(\LL_n)$. 
Observe  that  the action $\agl{\nu} \curvearrowright\Pi(\AA_{2n})$ preserves canonical idempotent elements and trivial stabilizers. 
Thus, by applying Theorem \ref{202405081500}, we have a poset isomorphism 
$$\twotilt\Pi(\AA_{2n}) \cong \twotilt \Pi(\LL_n) \cong \sttilt \Pi(\LL_{n}) \ \textnormal{and}\ \tilt \Pi(\AA_{2n}) \cong \silt \Pi(\LL_{n}). $$
Then the conclusion follows from Theorem \ref{tau-weyl}.
\end{proof}

\subsection{The generalized preprojective algebra $\Pi(\CC_{n})$ of type $\CC_{n}$ in the sense of \cite{GLS} ($\chara \kk =2$ case).}

In the previous section, we deal with the Nakayama twisted algebra $\Pi(\AA_{2n})*\agl{\nu}$. 
A natural question  is that how about $\Pi(\AA_{2n+1}) * \agl{\nu}$. 
The point here is that the stabilizer group $\stab(n)$ of the middle vertex $n$ is a group of order $2$ 
and its group algebra $\kk\stab(n)$ is  local if and only if $\chara \kk =2$. 

In this section, 
we deal with the Nakayama twisted algebra $\Pi(\AA_{2n+1} )* \agl{\nu}$ in the case $\chara \kk =2$.  
A key observation given in the following lemma tells that 
 $\Pi(\AA_{2n+1} )* \agl{\nu}$ is Morita equivalent to 
 the preprojective algebra $\Pi(\CC_{n})$ of type $\CC_{n}$ in the sense of Geiss-Leclerc-Schr\"{o}er in \cite{GLS}.  
 Thus the same method with the previous section determines  the posets of support $\tau$-tilting modules and 
 silting complexes over $\Pi(\CC_{n})$.

\begin{lemma}\label{202405221753}
Assume that  $\chara \kk=2$. 
Then the following assertions hold.

\begin{enumerate}[(1)]
\item The action $\agl{\nu} \curvearrowright \Pi(\AA_{2n+1})$ preserves the canonical idempotents and has local stabilizers. 

\item The skew group algebra $A = \Pi(\AA_{2n+1}) * \agl{\nu}$ is Morita equivalent to the generalized preprojective algebra $\Pi(\CC_{n})$ of type $\CC_{n}$ in the sense of \cite{GLS}. 
\end{enumerate}

\end{lemma} 

\begin{proof}
(1) follows from Lemma \ref{2023100319381}. 

(2) can be proved a direct computation similar to that in the proof of  Lemma \ref{202405221519}.  
\end{proof}

\begin{remark}
After finishing the first draft of this paper, R. Wang notified to us that she and X-W Chen \cite[Theorem B]{CW2} independently obtained 
a general result which gives a relationship between  skew group algebras of preprojective algebras (which is not necessarily the action on Nakayama automorphism) and 
the generalized preprojective algebras  in the sense of Geiss-Leclerc-Schr\"{o}er. 
We leave it to a future work to obtain a generalization of Theorem \ref{202405221756} below by applying  our results to the results of \cite{CW2}.
\end{remark}

As a consequence, we obtain the following result by the same argument of Theorem \ref{typeL}.

\begin{theorem}\label{202405221756}
Assume that $\chara \kk =2$. 
Let $\Pi(\CC_{n})$ be the preprojective algebra of type $\CC_n$. 
\begin{itemize}
\item[(1)] We have the following  anti-isomorphism of posets
$$W\to  \sttilt\Pi(\CC_n).$$
\item[(2)] We have the following  anti-isomorphism of posets
$$B\to  \silt \Pi(\CC_n).$$
\end{itemize}
\end{theorem}

\begin{remark}
Fu-Geng \cite{FG} studied $\tau$-tilting theory of the generalized preprojective algebras $\Pi(C,D)$ in the sense of \cite{GLS} over a field $\kk$ of arbitrary characteristic 
and
established a bijective correspondence 
between the support $\tau$-tilting modules over $\Pi(C,D)$ and the corresponding Weyl group $W(C)$. 

The poset isomorphism of (1) in the above theorem can be looked as an refinement of the bijection by Geng-Fu for $\Pi(\CC_{n})$ in the case $\chara \kk =2$.  
We can expect that  there are  the similar poset isomorphisms for any   $\Pi(C,D)$ over a base field of arbitrary characteristic.
\end{remark}

We note that if $\chara \kk \neq 2$, then the Nakayama twisted algebra $\Pi(\AA_{2n+1})*\agl{\nu}$ is Morita equivalent to 
the preprojective algebra of type $\CC_{n}$ given in \cite[Section 7]{ES}.

\subsection{The folded mesh algebra $K^{(m)}(\AA_{n})$ of type $\AA_{n}$}

Adachi-Kase \cite{Adachi-Kase} constructed two examples of  algebras which are tilting discrete but not silting discrete.  
One of them is given by
the $m$-fold mesh algebra $K^{(m)}(\AA_{n})$ of type $\AA_{n}$, which is called the stable Auslander algebra in \cite{Adachi-Kase}. 
Indeed, 
Adachi-Kase \cite[Theorem 4.1]{Adachi-Kase} proved that 
$K^{(m)}(\AA_{n})$  is tilting discrete, but not silting discrete, 
provided that $n, m\geq 5$, $n$ is odd  and $m$ is coprime to $n-1$. 

It is natural to ask how is the case that $n$ is even. 
The same method with  Adachi-Kase  works and  
 we can show that $K^{(m)}(\AA_{n})$ is not silting discrete even  in the case that $n$ is even.   
 As to tilting discreteness, Adachi-Kase proved it by detailed and long analysis of tilting modules over $K^{(m)}(\AA_{n})$. 
The aim of this section is that using our methods developed in the paper, we show that 
in the case that $n$ is even,  the tilting poset of $K^{(m)}(\AA_{n})$ is isomorphic to that of $\Pi(\AA_{n})$ 
and  $K^{(m)}(\AA_{n})$ is tilting discrete.

\begin{theorem}\label{202404101811}
Let $K := K^{(m)}(\AA_{n})$ of the $m$-fold mesh algebra of type $\AA_{n}$. 
Assume that $n$ is even and that $m$ is coprime to $n-1$. 
Then, 
there exists a poset isomorphism $\tilt K \xrightarrow{\cong } \tilt \Pi(\AA_{n})$ that sends $K[l]$ to $\Pi(\AA_{n})[l]$ for any $l\in \ZZ$.  
Consequently $K$ is tilting discrete. 
\end{theorem}

Before giving a proof, we recall basic facts about the folded mesh algebras. 
We basically follows the notation and conventions of \cite[Section 3, 4]{AJS: symmetry}.

Let $K(\ZZ\AA_{n})$ be the mesh algebra of the stable translation quiver $\ZZ \AA_{n}$ and $\tau$ the canonical translation of $\ZZ\AA_{n}$. 
Then recall that for a non-negative integer $m$,  
the $m$-folded mesh algebra $K^{(m)}(\AA_{n})$ is defined to be the orbit category $K^{(m)}(\AA_{n})  := K(\ZZ\AA_{n})/\agl{\tau^{m}}$ 
(regarded as an algebra).  
We  also recall that the $1$-folded mesh algebra $K^{(1)}(\AA_{n}) $ is nothing but the preprojective algebra $\Pi(\AA_{n})$.

We need the following lemma. 

\begin{lemma}\label{202405231145}
The Nakayama twisted algebra $K^{(m)}(\AA_{n})* \agl{\nu}$ is Morita equivalent to the orbit category $K(\ZZ\AA_{n})/\agl{\tau^{m}, \nu}$.
\end{lemma}

\begin{proof}
By general relationship between orbit categories and skew group algebras, given in \cite[Definition 3.6]{Asashiba}, 
we see that 
the Nakayama twisted algebra $K^{(m)}(\AA_{n})* \agl{\nu}$ is Morita equivalent to the orbit category $(K(\ZZ\AA_{n})/\agl{\tau^{m}})/\agl{\nu}$. 
On the other hand, 
by \cite[Theorem 4.2]{AJS: symmetry}, the Nakayama automorphism $\nu$ of $K^{(m)}(\AA_{n})$ is induced from that of $K(\ZZ\AA)$. 
It follows that 
$(K(\ZZ\AA_{n})/\agl{\tau^{m}})/\agl{\nu} \cong K(\ZZ\AA_{n})/\agl{\tau^{m}, \nu}$. 
\end{proof}

Now we proceed the proof of the main theorem of this section.

\begin{proof}[Proof of Theorem \ref{202404101811}]
First, we claim that the Nakayama twisted algebra $K' := K*\agl{\nu}$ is Morita equivalent to $\Pi(\AA_{n}) * \agl{\nu}$ and hence is Morita equivalent to $\Pi(\LL_{n/2})$. 
Indeed, 
from the explicit description of the Nakayama functor $\nu$ of $K(\ZZ\AA_{n})$ given in \cite[Theorem 4.2]{AJS: symmetry}, 
we observe that $\nu^{2} = \tau^{-n+1}$. 
Since $m$ and $n-1$ are coprime to each other, we have $\tau\in \agl{\tau^{m}, \nu}$. 
It follows that $\agl{\tau^{m}, \nu} = \agl{\tau, \nu}$ 
and that $K(\ZZ\AA_{n})/\agl{\tau^{m}, \nu} = K(\ZZ\AA_{n})/\agl{\tau,\nu}$. 
Thanks to Lemma \ref{202405231145} and the fact $K^{(1)}(\AA_{n}) \cong \Pi(\AA_{n})$, we can conclude that the desired statement holds. 

It follows from the claim and Theorem \ref{typeL} that $\silt K' \cong  \silt \Pi(\LL_{n/2}) \cong \tilt \Pi(\AA_{n})$.
We note that under this isomorphism, $K'$ corresponds to $\Pi(\AA_{n})$
Therefore, 
the interval $[K'[l], K']$ in  $(\silt K')^{\agl{\nu}\mod}\subset \silt K'$ is finite for any $l\geq 0$. 
Since
$\tilt K = (\silt K)^{\nu} \cong (\silt K')^{\agl{\nu}\mod}$, we see that $K$ is tilting discrete. 
By \cite[Theorem 4.2]{AJS: symmetry}, the action $\agl{\nu} \curvearrowright K$ preserves the canonical idempotent elements and has trivial stabilizers. 
Finally applying Theorem \ref{202403311841}, we conclude that all silting complexes over $K'$ is $\agl{\nu}\mod$-stable and
 $\tilt K$ is 
isomorphic to $\tilt \Pi(\AA_{n})$. 
\end{proof}

In a subsequent work, we study  $m$-fold mesh algebras of other types.

\appendix 
\section{Reinterpretation to  $G$-graded algebras}\label{202404011608}

In this section, 
we briefly explain that our results given in the main body of the paper 
reinterpreted to results for a finite dimensional $G$-graded algebra $A=\bigoplus_{g \in G}A_{g}$ via Cohen-Montgomery duality \cite{CM}. 

We  also explain that in the special case that $A$ is strongly $G$-graded, 
we can obtain results for the canonical inclusion   
$A_{e}\hookrightarrow A$ where $A_{e}$ denotes the degree $e$-part of $A$ (where $e$ is the unit element of $G$). 
These  are  generalizations of the results given in the paper \cite{BMM}  
in which the authors dealt with a finite dimensional strongly $G$-graded algebra $A$ and often assumed that $A_{e}$ is selfinjective and that $\chara \kk \nmid |G|$. Given a group $H$ and a normal subgroup $N$ of $H$, the group algebra $\kk H$ acquires a canonical $G:= H/N$-grading 
such that $(\kk H)_{e} = \kk N$. 
Thus, our results can be regarded as generalizations of results given in \cite{KK} for a group extension $N \triangleleft H$.

\subsection{$G$-graded algebras} 
Let $G$ be a finite group and 
 $A:= \bigoplus_{g \in G}A_{g}$ a finite dimensional $G$-graded algebra 
and $\grmod A$ the category of finite dimensional $G$-modules.

We denote the dual Hopf algebra of the group Hopf algebra $\kk G$ by  $\kk G^{*}:= \Hom_{\kk}(\kk G, \kk)$. 
By \cite[Proposition 1.3]{CM}, a finite dimensional $G$-graded algebra $A = \bigoplus_{g\in A}A_{g}$ may be  regarded as $\kk G^{*}$-module algebra. 
Let $\Lambda := A\# \kk G^{*}$ be the smash product \cite[p. 241]{CM}. 
Then by \cite[Theorem 2.2]{CM} there is an equivalence $\mod \Lambda \simeq \grmod A$. 

By \cite[Lemma 3.3]{CM}, the group $G$ acts on $\Lambda$. 
One of the main result of the paper \cite[Theorem 3.5(Duality for coactions)]{CM} tells  that 
the skew group algebra $\tilde{A} := \Lambda*G=(A\# \kk G^{*})*G$ is isomorphic to the matrix algebra $M_{|G|}(A)$ of the order $|G|$. 
In particular the algebra $\tilde{A}$ is Morita equivalent to $A$ and we have an equivalence 
$\mod A \simeq \tilde{A}\mod$. 

Then it is tedious but straightforward to check that 
the induction functor $\Ind_{\Lambda}^{\tilde{A}}: \mod \Lambda \to \tilde{A}\mod$ 
corresponds to 
the functor  $U: \grmod A \to \mod A$  which forgets the grading of $G$-graded $A$-modules 
via the equivalences given above.  
Namely the left diagram is commutative up to natural isomorphism:
\[
\begin{xymatrix}{
\mod \widetilde{A}  \ar@{-}[r]^{\sim } & \mod A \\
\mod \Lambda \ar[u]^{\Ind_{\Lambda}^{\widetilde{A}}} \ar@{-}[r]^{\sim} &
\grmod A \ar[u]_{U},
}\end{xymatrix}
\begin{xymatrix}{
\mod \widetilde{A}  \ar@{-}[r]^{\sim} \ar[d]_{\Res_{\Lambda}^{\tilde{A}}} & \mod A  \ar[d]^{V}\\
\mod \Lambda  \ar@{-}[r]^{\sim} &
\grmod A.
}\end{xymatrix}
\]

It is also straightforward to check that the functor $V: \mod A \to \grmod A$ defined below  
corresponds to the restriction functor $\Res_{\Lambda}^{\tilde{A}}: \tilde{A}\mod \to \mod \Lambda$ 
via the above equivalences: 
Let $N$ be an (ungraded) $A$-module, we define a graded $A$-module $V(N)$ in the following way. 
The underlying $\kk$-vector space of $V(N)$ is defined to be $\kk G \otimes N$. 
For $g \in G$, the degree $g$-component $V(N)_{g}$ is defined to be $V(N)_{g} := g\otimes N$. 
Finally the action of a homogeneous element $a_{h} \in A_{h}$ on $g \otimes n$ for $n\in N$ is given by 
\[
a_{h} \cdot_{V(N)} (g\otimes n) := hg \otimes a_{h}n. 
\]
The assignment $N \mapsto V(N)$ gives a functor $V: \mod A \to \grmod A$.

\begin{definition}\label{202309161711}
Let $N$ be an (ungraded) $A$-module and $X$ a $G$-module. 
We define an $A$-module $X \otimes_{\kk}^{G} N$ in the following way: 
The underlying $\kk$-vector space is defined to be $X \otimes_{\kk} N$. 
The action of a homogeneous element $a_{g} \in A_{g}$ is given by 
\[
a_{g} \cdot(x \otimes n) := gx \otimes a_{g}n. 
\]
\end{definition}

We note that $UV(N) = \kk G\otimes_{\kk}^{G} N$. 

\begin{definition}\label{202309191621}
An $A$-module $N$ is called ($\mod G$)-\emph{stable} 
if $X \otimes_{\kk}^{G} N \in \add N$ for any $G$-modules $X$.
\end{definition}

We explain $G$-stability of $G$-graded $A$-modules $M$. 
Let $ g\in G$. We define the degree shift $M(g)$ of a graded $A$-module $M=\bigoplus_{g \in G} M_{g}$ in the following way: 
The underlying $A$-module of $M(g)$ is defined to be that of $M$. 
The grading is defined by $(M(g))_{h} := M_{hg^{-1}}$. 
It is straightforward to check that under the identification of graded $A$-modules and $\Lambda= A\# \kk G^{*}$-modules 
the degree shift $M(g)$ corresponds to the twist ${}_{g} M$ given in Section \ref{202403291718}.
\[
\grmod A \leftrightarrow \mod \Lambda, \ M(g) \leftrightarrow {}_{g} M. 
\]

Thus we introduce the following terminology.

\begin{definition}\label{202309161819}
A graded $A$-module $M$ is called $G$-\emph{stable} 
if $M \cong M(g)$ for any $g \in G$.
\end{definition} 

Using these preparations, we can reinterpret the results in the main body of the paper  for 
$G$-graded and ungraded $A$-modules. 
Among other things, we point out the following reinterpretation of Theorem \ref{202309301601}. 
We denote by $(\grsttilt A)^{G}$ the poset of $G$-stable basic graded support $\tau$-tilting $\Lambda$-module, 
by $(\sttilt A)^{\mod G}$ the poset of \textup{(}$\mod G$\textup{)}-stable basic support $\tau$-tilting $A$-modules. 
Then

\begin{theorem}\label{202404010037}
The adjoint pair $(U,V)$ induces an isomorphism of posets 
\[
U: (\grsttilt A)^{G} \longleftrightarrow (\sttilt A)^{\mod G} :V.
\]
\end{theorem}

\subsection{Strongly graded case}

Now we discuss the spacial case that  a graded algebra $A = \bigoplus_{g \in G} A_{g}$ is strongly graded \cite[p.245]{CM}.  
In this section, we denote the degree $e$-subalgebra by $\Lambda :=A_{e}$ (where $e$ is the unit element of $G$).  
Then by \cite[Theorem 2.2, Theorem 2.12]{CM}, the category $\grmod A$ is equivalent to $\mod \Lambda$ 
via the functor $(-)_{e}: \grmod A \to \mod \Lambda$ that takes the degree $e$-part of graded $A$-modules. 
Its quasi-inverse is the functor $A\otimes_{\Lambda}-: \mod \Lambda \to \grmod A$. 
\begin{equation}\label{202309161906}
(-)_{e}: \grmod A \rightleftarrows \mod \Lambda :A\otimes_{\Lambda} -
\end{equation}
Let $g\in G$. Under the equivalence, the degree shift functor $(g)$ corresponds to the functor $A_{g^{-1}}\otimes_{\Lambda} -$. 

\begin{definition}\label{202309161916}
A $\Lambda$-module $M$ is called $G$-\emph{stable} 
if $A_{g^{-1}}\otimes_{\Lambda} M \cong M$ for any $g\in G$. 
\end{definition}

The composition  $U \circ (A\otimes_{\Lambda}- ): \mod \Lambda\to \mod A$ is nothing but the induction functor $\Ind_{\Lambda}^{A}$ 
and that  the composition $(-)_{e}\circ V: \mod A \to \mod \Lambda$ is the restriction functor $\Res_{\Lambda}^{A}$. 
\[
\Ind_{\Lambda}^{A} =U \circ (A\otimes_{\Lambda}-) : \mod \Lambda \rightleftarrows  \mod A: (-)_{e} \circ V = \Res_{\Lambda}^{A}.
\]

Using these preparations, we can specialize the results  for a $G$-graded algebra $A$ 
(that are obtained as reinterpretations of results in the main body)  
to  the induction and the reduction functors associated to the canonical injection $\Lambda \hookrightarrow A$. 
As a consequence, we obtain  generalizations of the results given in the paper \cite{BMM}  
in which the authors dealt with a finite dimensional strongly $G$-graded algebra $A$ and often assumed that $\Lambda=A_{e}$ is selfinjective and that $\chara \kk \nmid |G|$. 

Let $H$ be a group and $N$ a normal subgroup of $H$. We set $G:= H/N$. 
Then the group algebra $A:= \kk H$ becomes a strongly $G$-graded algebra such that $A_{e} = \kk N$ by setting $A_{g}:= \kk N \underline{g}$ for $g\in G$ where $\underline{g}$ denotes a representative of $g$. 
Thus, our results can be regarded as generalizations of results given in \cite{KK} for a group extension $N \triangleleft H$.

Among other things, we point out the following theorem obtained as a specialization of Theorem \ref{202404010037}.

\begin{theorem}\label{202404010043}
The adjoint pair $(\Ind_{\Lambda}^{A}, \Res_{\Lambda}^{A})$ induces an isomorphism of posets 
\[
\Ind_{\Lambda}^{A}: (\sttilt \Lambda)^{G} \longleftrightarrow (\sttilt A)^{\mod G} :\Res_{\Lambda}^{A}.
\]
\end{theorem}

\begin{remark}
Let $\Lambda$ be an algebra with an action  of $G$ and $A:= \Lambda *G$ as in the main body of the paper.  
Then if we define $A_{g} := \Lambda *g$ for $g \in G$,  then   $A:= \Lambda* G$ becomes a  strongly $G$-graded algebra. 
Hence, applying Theorem \ref{202404010043} to this case, we can  recover Theorem \ref{202309301601}.  

In a similar way, we can obtain results of the main body of the paper by first establishing results for a $G$-graded algebra and functors $U,V$.
\end{remark}

\section{A formalism by  theory of Hopf algebras and tensor categories}\label{202403292142}

In this section, 
we discuss a formal aspect of the paper 
including \textup{(}$\mod G$\textup{)}-action on $A=\Lambda * G$-modules given in Definition \ref{202309161735}, 
by using  theory of Hopf algebras and tensor categories. 
For unexplained terminology, we refer \cite{EGNO}.

%
%
%
%

Since the group algebra $\kk G$ has a canonical structure of a Hopf algebra, 
the module category $\mod G=\mod \kk G$ has a canonical structure of a symmetric monoidal category $\cC_{\kk G} := (\mod G, \otimes_{\kk}^{G}, \kk)$. 
(We remark that the tensor product of the monoidal category $\cC_{\kk G}$ is usually denoted by $\otimes_{\kk}$.)

A pair $\tilde{\Lambda} = (\Lambda, \rho)$ of an algebra $\Lambda$ and an action $\rho: G \curvearrowright \Lambda$ may be identified with 
an algebra object $\tilde{\Lambda}$ of  $\cC_{G}$.  
It is well-known that the category  $\tilde{\Lambda}\mod_{\cC_{\kk G}}$ of $\tilde{\Lambda}$-module objects in $\cC_{\kk G}$ 
is equivalent to the category  $\mod A$ of (ordinary) modules over the skew group algebra $A:= \Lambda * G$. 
On the other hand, 
the tensor product $\otimes_{\kk}^{G}$ of $\cC_{\kk G}$ 
induces a functor $\otimes_{\kk}^{G}: \cC_{\kk G} \times \mod_{\cC_{\kk G}} \tilde{\Lambda}  \to \mod_{\cC_{\kk G}}\tilde{\Lambda}$, which gives $\mod_{\cC_{\kk G}} \tilde{\Lambda}$ a structure of a $\cC_{\kk G}$-module category \cite[Proposition 7.8.10]{EGNO}.  
It is straightforward to check that under the equivalence $\mod_{\cC_{\kk G}}\tilde{\Lambda} \simeq \mod A$, 
this functor gives the tensor product $(X, N) \mapsto X \otimes_{\kk}^{G} N$ of Definition \ref{202309161735}.  

Let $\kk G^{*}:=\Hom_{\kk}(G, \kk)$ be the dual Hopf algebra of $\kk G$. 
An action $\rho: G\curvearrowright \Lambda$ may be regarded  as a map $\rho: \kk G \otimes \Lambda \to \Lambda$
that gives a $\kk G$-module algebra structure on the algebra $\Lambda$. 
The induced map $\rho^{*}: \Lambda \to \kk G^{*} \otimes \Lambda$ gives $\Lambda$  a structure of $\kk G^{*}$-comodule algebra. 
For a $\Lambda$-module $M$ and a $\kk G^{*}$-module $Y$, 
we can construct a $\Lambda$-module $Y\otimes_{\kk}^{\kk G^{*}} M$ in the following way: 
the underlying $\kk$-vector space is defined to be the usual tensor product $Y \otimes M$ of the underlying spaces. 
The multiplication $\Lambda \otimes( Y \otimes_{\kk}^{\kk G^{*}} M) \to  Y \otimes_{\kk}^{\kk G^{*}} M$ is defined by the following 
composition 
\[
\Lambda \otimes( Y \otimes M)
 \xrightarrow{ \rho^{*} \otimes \id} 
 \kk G^{*} \otimes \Lambda \otimes( Y \otimes M) 
 \cong (\kk G^{*} \otimes Y) \otimes (\Lambda \otimes M) \to Y \otimes M
\]
where the third map is the tensor product of the action maps of $Y$ and $M$. 
Let $\cC_{\kk G^{*}}=(\mod \kk G^{*}, \otimes_{\kk}^{\kk G^{*}}, \kk)$ be the symmetric monoidal category of $\kk G^{*}$-modules. 
Then the induced functor $\otimes_{\kk}^{\kk G^{*}}: \cC_{\kk G^{*}} \times \mod \Lambda \to \mod \Lambda$ gives 
$\mod \Lambda$ a structure of $\cC_{\kk G^{*}}$-module category.

The dual basis $\{ g^{*} \}_{ g \in G} \subset \kk G^{*}$  of the canonical basis $\{g \}_{g\in G} \subset \kk G$
is a complete set of primitive orthogonal idempotent elements of $\kk G^{*}$ 
and hence $\kk G^{*} \cong \prod_{g \in G}\kk g^{*} \cong \kk^{\times |G|}$ is semi-simple as an algebra.
Thus a $\kk G^{*}$-module $Y$ is a direct sum of simple $\kk G^{*}$-modules $\kk g^{*}$ for $g \in G$ 
and the tensor product $Y \otimes_{\kk}^{\kk G^{*}}-$ is completely determined by $\kk g^{*} \otimes_{\kk}^{\kk G^{*}}-$ for $g \in G$.
Let $g \in G$ and $M$ a $\Lambda$-module. 
Then the $\Lambda$-module  $\kk g^{*} \otimes_{\kk}^{\kk G^{*}} M$ is isomorphic to ${}_{g} M$. 
Hence, we see that $G$-stability  for basic modules coincides with 
($\mod\kk G^{*}$)-stability (defined  in the same way of Definition \ref{202309161735}). 
We understand  that the difference between 
the characterization of $G$-stability (Corollary \ref{202310021718}) and that of \textup{(}$\mod G$\textup{)}-stability (Proposition \ref{202309150913}) is caused by the fact that $\kk G^{*}$ is semi-simple as an algebra but $\kk G$ is not in general. 
Finally,  our main results Theorem \ref{202309301357}  and Theorem \ref{202403121806}  can be looked as bijections between 
($\mod \kk G^{*}$)-stable objects and ($\mod \kk G$)-stable objects. 

It is worth pointing a possible generalization obtained by replacing the group Hopf algebra $\kk G$ 
with a (finite dimensional) Hopf algebra $H$. 
The above constructions works for an  $H$-modules algebra $\Lambda$ and the smash product algebra $A:= \Lambda\# H$. 
Thus we can expect that we have bijections between ($\mod H$)-stable objects and ($\mod H^{*}$)-stable objects. 
(Furthermore, we can expect more general statements by using tensor categories as in \cite{DHL}.)
Some of the result of the paper may be easily generalized for a general Hopf algebra $H$ and its module algebra $\Lambda$. 
However, we remark that 
one of the key step Lemma \ref{202309151536}(2)  does not hold true for a general Hopf algebra $H$ (see, e.g., 
\cite[Theorem 3.2]{CM}, \cite[Corollary 9.3.4]{Montgomery}). 
Thus, we decided to leave generalizations for Hopf algebras to our future research.

\end{document}